\documentclass[11pt]{article}
\usepackage{amsmath}
\usepackage{amssymb}
\usepackage{theorem}
\usepackage{pstricks}
\usepackage{euscript}
\usepackage{epic,eepic}
\PassOptionsToPackage{normalem}{ulem}
\newcommand{\minimize}[2]{\ensuremath{\underset{\substack{{#1}}}%
{\mathrm{minimize}}\;\;#2 }}

\newcommand{\menge}[2]{\big\{{#1} \mid {#2}\big\}}

\newcommand{\emp}{\ensuremath{{\varnothing}}}

\newcommand{\scal}[2]{\left\langle{#1}\mid {#2} \right\rangle}

\newcommand{\exi}{\ensuremath{\exists\,}}
   
\newcommand{\HH}{\ensuremath{\mathcal H}}
\newcommand{\GG}{\ensuremath{\mathcal G}}
\newcommand{\bary}{\ensuremath{\widetilde{\boldsymbol{y}}}}
\newcommand{\barp}{\ensuremath{\widetilde{\boldsymbol{p}}}}
\newcommand{\barq}{\ensuremath{\widetilde{\boldsymbol{q}}}}

\newcommand{\BL}{\ensuremath{\EuScript B}\,}
\newcommand{\HHH}{\ensuremath{\boldsymbol{\mathcal H}}}
\newcommand{\KKK}{\ensuremath{\boldsymbol{\mathcal K}}}
\newcommand{\GGG}{\ensuremath{\boldsymbol{\mathcal G}}}

\newcommand{\sri}{\ensuremath{\operatorname{sri}}}

\newcommand{\RP}{\ensuremath{\left[0,+\infty\right[}}
\newcommand{\RPP}{\ensuremath{\,\left]0,+\infty\right[}}
\newcommand{\RX}{\ensuremath{\,\left]-\infty,+\infty\right]}}
\newcommand{\NN}{\ensuremath{\mathbb N}}
\newcommand{\dom}{\ensuremath{\operatorname{dom}}}

\newcommand{\gr}{\ensuremath{\operatorname{gra}}}
\newcommand{\prox}{\ensuremath{\operatorname{prox}}}
\newcommand{\inte}{\ensuremath{\operatorname{int}}}
\newcommand{\cart}{\ensuremath{\mbox{\huge{$\times$}}}}

\newcommand{\RPX}{\ensuremath{{[0,\pinf]}}}

\newcommand{\Argmin}{\ensuremath{\operatorname{Argmin}}}
\newcommand{\ran}{\ensuremath{\operatorname{ran}}}
\newcommand{\zer}{\ensuremath{\operatorname{zer}}}

\newcommand{\Id}{\ensuremath{\operatorname{Id}}}

\newcommand{\weakly}{\ensuremath{\rightharpoonup}}

\newcommand{\pinf}{\ensuremath{+\infty}}
\newtheorem{theorem}{Theorem}[section]
\newtheorem{lemma}[theorem]{Lemma}

\newtheorem{corollary}[theorem]{Corollary}
\newtheorem{proposition}[theorem]{Proposition}
\newtheorem{definition}[theorem]{Definition}
\theoremstyle{plain}{\theorembodyfont{\rmfamily}
}
\theoremstyle{plain}{\theorembodyfont{\rmfamily}
}
\theoremstyle{plain}{\theorembodyfont{\rmfamily}
}
\theoremstyle{plain}{\theorembodyfont{\rmfamily}
}
\theoremstyle{plain}{\theorembodyfont{\rmfamily}
\newtheorem{problem}[theorem]{Problem}}
\theoremstyle{plain}{\theorembodyfont{\rmfamily}
\newtheorem{remark}[theorem]{Remark}}
\theoremstyle{plain}{\theorembodyfont{\rmfamily}
}

\numberwithin{equation}{section}

\begin{document}

\title{\sffamily\huge A Monotone{\,\LARGE$+$\,}Skew 
Splitting Model for\\
Composite Monotone Inclusions in Duality\footnote{Contact author: 
P. L. Combettes, {\ttfamily plc@math.jussieu.fr},
phone: +33 1 4427 6319, fax: +33 1 4427 7200.
This work was supported by the Agence Nationale de la 
Recherche under grant ANR-08-BLAN-0294-02.}}

\author{Luis M. Brice\~{n}o-Arias$^{1,2}$ and 
Patrick L. Combettes$^1$\\[5mm]
\small $\!^1$UPMC Universit\'e Paris 06\\
\small Laboratoire Jacques-Louis Lions -- UMR CNRS 7598\\
\small 75005 Paris, France\\[4mm]
\small $\!^2$UPMC Universit\'e Paris 06\\
\small \'Equipe Combinatoire et Optimisation -- UMR CNRS 7090\\
\small 75005 Paris, France 
}

\date{~}

\maketitle

\vskip 8mm

\begin{abstract} \noindent
The principle underlying this paper is the basic observation
that the problem of simultaneously solving a large class of 
composite monotone inclusions and their duals can be reduced to 
that of finding a zero of the sum of a maximally monotone operator 
and a linear skew-adjoint operator. 
An algorithmic framework is developed 
for solving this generic problem in a Hilbert space setting. 
New primal-dual splitting 
algorithms are derived from this framework for inclusions 
involving composite monotone operators, and convergence results are 
established.
These algorithms draw their simplicity and efficacy from the
fact that they operate in a fully decomposed fashion in the sense 
that the monotone operators and the linear transformations 
involved are activated separately at each iteration.
Comparisons with existing methods are made and applications to 
composite variational problems are demonstrated.
\end{abstract}

{\small 

}

\newpage

\section{Introduction}
A wide range of problems in areas such as optimization, variational 
inequalities, partial differential equations, mechanics, economics, 
signal and image processing, or traffic theory can be reduced to 
solving inclusions involving monotone set-valued operators in a 
Hilbert space $\HH$, say
\begin{equation}
\label{e:basic}
\text{find}\;\;x\in\HH\quad\text{such that}\quad 
z\in Mx,
\end{equation}
where $M\colon\HH\to 2^{\HH}$ is monotone and $z\in\HH$, e.g., 
\cite{Banf09,Smms05,Facc03,Fuku96,Glow89,Merc79,Roc76a,Tsen90,%
Tsen91,ZeidXX}.
In many formulations of this type, the operator $M$ can be 
expressed as the sum of two monotone operators, 
one of which is the composition of a monotone operator with a 
linear transformation and its adjoint. 
In such situations, it is often desirable to also 
solve an associated dual inclusion
\cite{Aldu05,Atto96,Ther96,Ecks99,Gaba83,Merc80,Mosc72,Penn00,%
Robi99,Robi01,Rock67}. 
The present paper is concerned with the numerical solution of such 
composite inclusion problems in duality. More formally, the basic 
problem we consider is the following. 

\begin{problem}
\label{prob:1}
Let $\HH$ and $\GG$ be two real Hilbert spaces, 
let $A\colon\HH\to 2^{\HH}$ and $B\colon\GG\to 2^{\GG}$ be 
maximally monotone, let $L\colon\HH\to\GG$ be linear and bounded,
let $z\in\HH$, and let $r\in\GG$.
The problem is to solve the primal inclusion
\begin{equation}
\label{e:primal}
\text{find}\;\;x\in\HH\quad\text{such that}\quad 
z\in Ax+L^*B(Lx-r)
\end{equation}
together with the dual inclusion
\begin{equation}
\label{e:dual}
\text{find}\;\;v\in\GG\quad\text{such that}\quad 
-r\in -LA^{-1}(z-L^*v)+B^{-1}v.
\end{equation}
The set of solutions to \eqref{e:primal} is denoted by 
${\mathcal P}$ and the set of solutions to \eqref{e:dual} by 
${\mathcal D}$.
\end{problem}

A classical instance of the duality scheme described in 
Problem~\ref{prob:1} is the Fenchel-Rockafellar framework 
\cite{Rock67} which, under suitable constraint qualification, 
corresponds to letting $A$ and $B$ be subdifferentials of proper 
lower semicontinuous convex functions $f\colon\HH\to\RX$ and 
$g\colon\GG\to\RX$, respectively. 
In this scenario, the problems in duality are 
\begin{equation}
\label{e:2010-11-18f}
\minimize{x\in\HH}{f(x)+g(Lx-r)-\scal{x}{z}}
\end{equation}
and 
\begin{equation}
\label{e:2010-11-18g}
\minimize{v\in\GG}{f^*(z-L^*v)+g^*(v)+\scal{v}{r}}.
\end{equation}
Extensions of the Fenchel-Rockafellar framework to variational 
inequalities were considered in \cite{Aldu05,Ekel99,Gaba83,Mosc72}, 
while extensions to saddle function problems were proposed in 
\cite{Mcli74}. On the other hand, general monotone operators were 
investigated in \cite{Atto96,Ther96,Reic05,Merc80} in the case 
when $\GG=\HH$ and $L=\Id$. The general duality setting 
described in Problem~\ref{prob:1} appears in 
\cite{Ecks99,Penn00,Robi99}.

Our objective is to devise an algorithm which solves \eqref{e:primal}
and \eqref{e:dual} simultaneously, and which uses the operators $A$, 
$B$, and $L$ separately. In the literature, several splitting 
algorithms are available for solving the primal problem 
\eqref{e:primal}, but they are restricted by stringent hypotheses.
Let us set 
\begin{equation}
A_1\colon\HH\to 2^{\HH}\colon x\mapsto -z+Ax\quad\text{and}\quad
A_2\colon\HH\to 2^{\HH}\colon x\mapsto L^*B(Lx-r), 
\end{equation}
and observe that solving \eqref{e:primal} is equivalent to
finding a zero of $A_1+A_2$.
If $B$ is single-valued and cocoercive (its inverse is strongly
monotone), then so is $A_2$, and \eqref{e:primal} can be solved by 
the forward-backward algorithm \cite{Opti04,Merc79,Tsen91}. 
If $B$ is merely Lipschitzian, or even just continuous, 
so is $A_2$, and \eqref{e:primal} can then be solved
via the algorithm proposed in \cite{Tsen00}. These algorithms
employ the resolvent of $A_1$, which is easily derived from that 
of $A$, and explicit applications of $A_2$, i.e., of $B$ and $L$.
They are however limited in scope by the fact that $B$ must
be single-valued and smooth. The main splitting algorithm to find
a zero of $A_1+A_2$ when both operators are set-valued 
is the Douglas-Rachford algorithm \cite{Joca09,Ecks92,Lion79,Svai10}.
This algorithm requires that both operators be maximally monotone and
that their resolvents be computable to within some quantifiable 
error. Unfortunately, these conditions are seldom met in the 
present setting since $A_2$ may not be maximally monotone 
\cite{Penn00,Robi99} and, more importantly, since there is no 
convenient rule to compute the resolvent of $A_2$ in terms of 
$L$ and the resolvent of $B$ unless stringent conditions are 
imposed on $L$ (see \cite[Proposition~23.23]{Livre1} and 
\cite{Fuku96}). 

Our approach is motivated by the classical Kuhn-Tucker theory
\cite{Rock74}, which asserts that points $\overline{x}\in\HH$ 
and $\overline{v}\in\GG$ satisfying the conditions 
\begin{equation}
\label{e:2010-11-18a}
(0,0)\in\big(-z+\partial f(\overline{x})+L^*\overline{v},\,
r+\partial g^*(\overline{v})-L\overline{x}\big)
\end{equation}
are solutions to \eqref{e:2010-11-18f} and \eqref{e:2010-11-18g}, 
respectively. By analogy, it is natural to consider 
the following problem in conjunction with Problem~\ref{prob:1}. 

\begin{problem}
\label{prob:2}
In the setting of Problem~\ref{prob:1}, 
let $\KKK=\HH\oplus\GG$ and set
\begin{equation}
\label{e:nyc2010-10-31z}
\boldsymbol{M}\colon\KKK\to 2^{\KKK}\colon(x,v)\mapsto
(-z+Ax)\times(r+B^{-1}v)\quad\text{and}\quad
\boldsymbol{S}\colon\KKK\to\KKK\colon(x,v)\mapsto(L^*v,-Lx).
\end{equation}
The problem is to 
\begin{equation}
\label{e:trial}
\text{find}\;\;\boldsymbol{x}\in\KKK\quad\text{such that}\quad 
\boldsymbol{0}\in\boldsymbol{M}\boldsymbol{x}
+\boldsymbol{S}\boldsymbol{x}.
\end{equation}
\end{problem}

The investigation of this companion problem may have various
purposes \cite{Aldu05,Ecks99,Penn00,Robi99}.
Ours is to exploit its simple structure to derive a new 
splitting algorithm to solve efficiently Problem~\ref{prob:1}. 
The crux of our approach is the simple observation that 
\eqref{e:trial} reduces the original primal- dual problem
\eqref{e:primal}--\eqref{e:dual} to 
that of finding a zero of the sum of a maximally monotone operator 
$\boldsymbol{M}$ and a bounded linear skew-adjoint transformation 
$\boldsymbol{S}$. In Section~\ref{sec:2} we establish the
convergence of an inexact splitting algorithm proposed in its 
original form in \cite{Tsen00}. Each iteration of this 
forward-backward-forward scheme performs successively an 
explicit step on $\boldsymbol{S}$, an implicit step on 
$\boldsymbol{M}$, and another explicit step on $\boldsymbol{S}$.
We then review the tight connections existing between 
Problem~\ref{prob:1} and Problem~\ref{prob:2}
and, in particular, the fact that solving the latter 
provides a solution to the former. In Section~\ref{sec:3},
we apply the forward-backward-forward algorithm to the 
monotone{\small$+$}skew Problem~\ref{prob:2} and obtain a new 
type of splitting algorithm for solving \eqref{e:primal} and 
\eqref{e:dual} simultaneously.
The main feature of this scheme, that distinguishes it from existing
techniques, is that at each iteration it employs the operators $A$, 
$B$, and $L$ separately without requiring any additional assumption 
to those stated above except, naturally, existence of solutions.
Using a product space technique, we then obtain a parallel splitting
method for solving the $m$-term inclusion
\begin{equation}
\label{e:2010-11-18x}
\text{find}\;\; x\in\HH\quad\text{such that}\quad
z\in\sum_{i=1}^mL_i^*B_i(L_ix-r_i),
\end{equation}
where each maximally monotone operator $B_i$ acts on a Hilbert 
space $\GG_i$, $r_i\in\GG_i$, and $L_i\colon\HH\to\GG_i$ is 
linear and bounded. Applications to variational problems are 
discussed in Section~\ref{sec:4}, where we provide a proximal
splitting scheme for solving the primal dual problem
\eqref{e:2010-11-18f}--\eqref{e:2010-11-18g}, as well as one 
for minimizing the sum of $m$ composite functions.

{\bfseries Notation.}
We denote the scalar products of $\HH$ and $\GG$
by $\scal{\cdot}{\cdot}$ and the associated norms by $\|\cdot\|$.
$\BL(\HH,\GG)$ is the space of bounded linear operators from $\HH$ 
to $\GG$, $\BL(\HH)=\BL(\HH,\HH)$, and the symbols $\weakly$ and 
$\to$ denote respectively weak and strong convergence. 
Moreover, $\HH\oplus\GG$ denotes the Hilbert direct sum of $\HH$
and $\GG$.
The projector onto a nonempty closed convex 
set $C\subset\HH$ is denoted by $P_C$, and 
its normal cone operator by $N_C$, i.e., 
\begin{equation}
\label{e:normalcone}
N_C\colon\HH\to 2^{\HH}\colon x\mapsto
\begin{cases}
\menge{u\in\HH}{(\forall y\in C)\;\;\scal{y-x}{u}\leq 0},
&\text{if}\;x\in C;\\
\emp,&\text{otherwise.}
\end{cases}
\end{equation}
Let $M\colon\HH\to 2^{\HH}$ be a set-valued operator.
We denote by $\ran M=\menge{u\in\HH}{(\exi x\in\HH)\;u\in Mx}$ 
the range of $M$, by $\dom M=\menge{x\in\HH}{Mx\neq\emp}$ its
domain, by $\zer M=\menge{x\in\HH}{0\in Mx}$ its set of zeros, 
by $\gr M=\menge{(x,u)\in\HH\times\HH}{u\in Mx}$ its 
graph, and by $M^{-1}$ its inverse, i.e., the operator with graph
$\menge{(u,x)\in\HH\times\HH}{u\in Mx}$. The resolvent of $M$ is
$J_M=(\Id+M)^{-1}$. Moreover, $M$ is monotone if
\begin{equation}
(\forall (x,y)\in\HH\times\HH)(\forall(u,v)\in Mx\times My)\quad
\scal{x-y}{u-v}\geq 0,
\end{equation}
and maximally so if there exists no monotone operator 
$\widetilde{M}\colon\HH\to 2^{\HH}$ such that $\gr M\subset\gr
\widetilde{M}\neq\gr M$.
In this case, $J_M$ is a nonexpansive operator defined everywhere
in $\HH$. For background on convex analysis and monotone operator 
theory, the reader is referred to \cite{Livre1,Zali02}.

\section{Preliminary results}
\label{sec:2}

\subsection{Technical facts}

The following results will be needed subsequently.

\begin{lemma}{\rm\cite[Lemma~3.1]{Else01}}
\label{l:1}
Let $(\alpha_n)_{n\in\NN}$ be a sequence in $\RP$, 
let $(\beta_n)_{n\in\NN}$ be a sequence in $\RP$, and let
$(\varepsilon_n)_{n\in\NN}$ be a summable sequence in $\RP$ such
that $(\forall n\in\NN)$ 
$\alpha_{n+1}\leq\alpha_n-\beta_n+\varepsilon_n$.
Then $(\alpha_n)_{n\in\NN}$ converges and 
$(\beta_n)_{n\in\NN}$ is summable.
\end{lemma}

\begin{lemma}{\rm\cite[Theorem~3.8]{Else01}}
\label{l:2}
Let $C$ be a nonempty subset of $\HH$ and let $(x_n)_{n\in\NN}$ 
be a sequence in $\HH$. Suppose that, for every $x\in C$, there 
exists a summable sequence $(\varepsilon_n)_{n\in\NN}$ in $\RP$ 
such that
\begin{equation}
\label{e:defqfejerIII}
(\forall n\in\NN)\;\;\|x_{n+1}-x\|^2\leq\|x_n-x\|^2+\varepsilon_n,
\end{equation}
and that every sequential weak cluster point of $(x_n)_{n\in\NN}$
is in $C$. Then $(x_n)_{n\in\NN}$ converges weakly to a point in
$C$.
\end{lemma}

\begin{definition}{\rm\cite[Definition~2.3]{Sico10}}
\label{d:demir}
An operator $M\colon\HH\to 2^{\HH}$ is \emph{demiregular} at
$x\in\dom M$ if, for every sequence $((x_n,u_n))_{n\in\NN}$ in 
$\gr M$ and every $u\in Mx$ such that $x_n\weakly x$ and 
$u_n\to u$, we have $x_n\to x$.
\end{definition}

\begin{lemma}{\rm\cite[Proposition~2.4]{Sico10}}
\label{l:2009-09-20}
Let $M\colon\HH\to 2^{\HH}$ and let $x\in\dom M$.
Then $M$ is demiregular at $x$ in each of the following cases.
\begin{enumerate}
\item
\label{l:2009-09-20i}
$M$ is uniformly monotone at $x$, i.e., there exists 
an increasing function $\phi\colon\RP\to\RPX$ that vanishes only 
at $0$ such that
$(\forall u\in Mx)(\forall (y,v)\in\gr M)$
$\scal{x-y}{u-v}\geq\phi(\|x-y\|)$.
In particular, $M$ is uniformly monotone, i.e.,
these inequalities hold for every $x\in\dom M$ and, a fortiori, $M$
is $\alpha$-strongly monotone, i.e., $M-\alpha\Id$ is monotone
for some $\alpha\in\RPP$.
\item
\label{l:2009-09-20iv-}
$J_M$ is compact, i.e., for every bounded set $C\subset\HH$,
the closure of $J_M(C)$ is compact. In particular, 
$\dom M$ is boundedly relatively compact, i.e., the intersection of 
its closure with every closed ball is compact.
\item
\label{l:2009-09-20vi}
$M\colon\HH\to\HH$ is single-valued with a single-valued continuous
inverse.
\item
\label{l:2009-09-20vii}
$M$ is single-valued on $\dom M$ and $\Id-M$, i.e., 
for every bounded sequence $(x_n)_{n\in\NN}$ in $\dom M$ such 
that $(Mx_n)_{n\in\NN}$ converges strongly, 
$(x_n)_{n\in\NN}$ admits a strong cluster point.
\end{enumerate}
\end{lemma}

\subsection{An inexact forward-backward-forward algorithm}

Our algorithmic framework will hinge on the following splitting 
algorithm, which was proposed in the error-free case in 
\cite{Tsen00}. We provide an analysis of the asymptotic behavior 
of a inexact version of this method which is of interest in its 
own right.

\begin{theorem}
\label{t:broome-st2010-10-27}
Let $\HHH$ be a real Hilbert space,
let $\boldsymbol{A}\colon\HHH\to 2^{\HHH}$ be maximally monotone,
and let $\boldsymbol{B}\colon\HHH\to\HHH$ be monotone.
Suppose that $\zer(\boldsymbol{A}+\boldsymbol{B})\neq\emp$ and that
$\boldsymbol{B}$ is $\beta$-Lipschitzian
for some $\beta\in\RPP$.
Let $(\boldsymbol{a}_n)_{n\in\NN}$, $(\boldsymbol{b}_n)_{n\in\NN}$, 
and $(\boldsymbol{c}_n)_{n\in\NN}$ be sequences 
in $\HHH$ such that 
\begin{equation}
\label{e:broome-st2010-10-27g}
\sum_{n\in\NN}\|\boldsymbol{a}_n\|<\pinf,\quad
\sum_{n\in\NN}\|\boldsymbol{b}_n\|<\pinf,
\quad\text{and}\quad
\sum_{n\in\NN}\|\boldsymbol{c}_n\|<\pinf, 
\end{equation}
let $\boldsymbol{x}_0\in\HHH$, let 
$\varepsilon\in\left]0,1/(\beta+1)\right[$,
let $(\gamma_n)_{n\in\NN}$ be a sequence in 
$[\varepsilon,(1-\varepsilon)/\beta]$,
and set
\begin{equation}
\label{e:rio2010-10-11u}
(\forall n\in\NN)\quad 
\begin{array}{l}
\left\lfloor
\begin{array}{l}
\boldsymbol{y}_n=\boldsymbol{x}_n-
\gamma_n(\boldsymbol{B}\boldsymbol{x}_n+\boldsymbol{a}_n)\\
\boldsymbol{p}_n=J_{\gamma_n\boldsymbol{A}}\,\boldsymbol{y}_n
+\boldsymbol{b}_n\\
\boldsymbol{q}_n=\boldsymbol{p}_n-\gamma_n(\boldsymbol{B}
\boldsymbol{p}_n+\boldsymbol{c}_n)\\
\boldsymbol{x}_{n+1}=\boldsymbol{x}_n-
\boldsymbol{y}_n+\boldsymbol{q}_n.
\end{array}
\right.\\[2mm]
\end{array}
\end{equation}
Then the following hold for some
$\boldsymbol{\overline{x}}\in\zer(\boldsymbol{A}+\boldsymbol{B})$.
\begin{enumerate}
\item
\label{t:broome-st2010-10-27i}
$\sum_{n\in\NN}\|\boldsymbol{x}_n-\boldsymbol{p}_n\|^2<\pinf$ and
$\sum_{n\in\NN}\|\boldsymbol{y}_n-\boldsymbol{q}_n\|^2<\pinf$.
\item
\label{t:broome-st2010-10-27ii}
$\boldsymbol{x}_n\weakly\boldsymbol{\overline{x}}$ and 
$\boldsymbol{p}_n\weakly\boldsymbol{\overline{x}}$. 
\item
\label{t:broome-st2010-10-27iii}
Suppose that one of the following is satisfied.
\begin{enumerate}
\item
\label{t:broome-st2010-10-27iiia}
$\boldsymbol{A}+\boldsymbol{B}$ is demiregular at
$\boldsymbol{\overline{x}}$.
\item
\label{t:broome-st2010-10-27iiib}
$\boldsymbol{A}$ or $\boldsymbol{B}$ is uniformly 
monotone at $\boldsymbol{\overline{x}}$.
\item
\label{t:broome-st2010-10-27iiic}
$\inte\zer(\boldsymbol{A}+\boldsymbol{B})\neq\emp$.
\end{enumerate}
Then $\boldsymbol{x}_n\to\boldsymbol{\overline{x}}$ and 
$\boldsymbol{p}_n\to\boldsymbol{\overline{x}}$. 
\end{enumerate}
\end{theorem}
\begin{proof}
Let us set 
\begin{equation}
\label{e:broome-st2010-10-29a}
(\forall n\in\NN)\quad 
\begin{cases}
\bary_n=\boldsymbol{x}_n-\gamma_n\boldsymbol{B}\boldsymbol{x}_n\\
\barp_n=J_{\gamma_n\boldsymbol{A}}\,\bary_n\\
\barq_n=\barp_n-\gamma_n\boldsymbol{B}\barp_n
\end{cases}
\quad\text{and}\quad
\begin{cases}
\boldsymbol{u}_n=\gamma_n^{-1}(\boldsymbol{x}_n-\barp_n)
+\boldsymbol{B}\barp_n-\boldsymbol{B}\boldsymbol{x}_n\\
\boldsymbol{e}_n=\boldsymbol{y}_n-\boldsymbol{q}_n-\bary_n+\barq_n.
\end{cases}
\end{equation}
Then 
\begin{equation}
\label{e:rio2010-10-12c}
(\forall n\in\NN)\quad 
\bary_n-\barp_n\in\gamma_n \boldsymbol{A}\barp_n
\quad\text{and}\quad
\boldsymbol{u}_n=\gamma_n^{-1}(\bary_n-\barp_n)
+\boldsymbol{B}\barp_n\in\boldsymbol{A}\barp_n
+\boldsymbol{B}\barp_n.
\end{equation}
Now let $\boldsymbol{x}\in\zer(\boldsymbol{A}+\boldsymbol{B})$ 
and let $n\in\NN$. We first note that
$(\boldsymbol{x},-\gamma_n \boldsymbol{B}\boldsymbol{x})
\in\gr\gamma_n \boldsymbol{A}$.
On the other hand, \eqref{e:rio2010-10-12c} yields
$(\barp_n,\bary_n-\barp_n)\in\gr\gamma_n\boldsymbol{A}$. 
Hence, by monotonicity of $\gamma_n\boldsymbol{A}$, 
$\scal{\barp_n-\boldsymbol{x}}{\barp_n-\bary_n-\gamma_n
\boldsymbol{B}\boldsymbol{x}}\leq 0$.
However, by monotonicity of $\boldsymbol{B}$, 
$\scal{\barp_n-\boldsymbol{x}}{\gamma_n
\boldsymbol{B}\boldsymbol{x}-\gamma_n\boldsymbol{B}
\barp_n}\leq 0$.
Upon adding these two inequalities, we obtain
$\scal{\barp_n-\boldsymbol{x}}{\barp_n-\bary_n-\gamma_n
\boldsymbol{B}\barp_n}\leq 0$.
In turn, we derive from \eqref{e:broome-st2010-10-29a} that
\begin{align}
\label{e:rio2010-10-11x}
2\gamma_n\scal{\barp_n-\boldsymbol{x}}{\boldsymbol{B}\boldsymbol{x}_n
-\boldsymbol{B}\barp_n}
&=2\scal{\barp_n-\boldsymbol{x}}{\barp_n-\bary_n-\gamma_n
\boldsymbol{B}\barp_n}
+2\scal{\barp_n-\boldsymbol{x}}{\gamma_n \boldsymbol{B}
\boldsymbol{x}_n+\bary_n-\barp_n}\nonumber\\
&\leq 2\scal{\barp_n-\boldsymbol{x}}{\gamma_n 
\boldsymbol{B}\boldsymbol{x}_n+\bary_n-\barp_n}\nonumber\\
&=2\scal{\barp_n-\boldsymbol{x}}{\boldsymbol{x}_n-\barp_n}\nonumber\\
&=\|\boldsymbol{x}_n-\boldsymbol{x}\|^2-\|\barp_n-\boldsymbol{x}\|^2
-\|\boldsymbol{x}_n-\barp_n\|^2
\end{align}
and, therefore, using the Lipschitz continuity of $\boldsymbol{B}$, 
that
\begin{align}
\label{e:rio2010-10-11y}
\|\boldsymbol{x}_n-\bary_n+\barq_n-\boldsymbol{x}\|^2
&=\|(\barp_n-\boldsymbol{x})
+\gamma_n(\boldsymbol{B}\boldsymbol{x}_n-
\boldsymbol{B}\barp_n)\|^2\nonumber\\
&=\|\barp_n-\boldsymbol{x}\|^2+
2\gamma_n\scal{\barp_n-\boldsymbol{x}}
{\boldsymbol{B}\boldsymbol{x}_n-\boldsymbol{B}\barp_n}+\gamma_n^2
\|\boldsymbol{B}\boldsymbol{x}_n
-\boldsymbol{B}\barp_n\|^2\nonumber\\
&\leq\|\boldsymbol{x}_n-\boldsymbol{x}\|^2
-\|\boldsymbol{x}_n-\barp_n\|^2+
\gamma_n^2\|\boldsymbol{B}\boldsymbol{x}_n-\boldsymbol{B}
\barp_n\|^2\nonumber\\
&\leq\|\boldsymbol{x}_n-\boldsymbol{x}\|^2-
(1-\gamma_n^2\beta^2)\|\boldsymbol{x}_n-\barp_n\|^2\nonumber\\
&\leq\|\boldsymbol{x}_n-\boldsymbol{x}\|^2-\varepsilon^2\|
\boldsymbol{x}_n-\barp_n\|^2.
\end{align}
We also derive from \eqref{e:rio2010-10-11u} and 
\eqref{e:broome-st2010-10-29a} the following inequalities. First, 
\begin{equation}
\label{e:broome-st2010-10-29c}
\|\bary_n-\boldsymbol{y}_n\|
=\gamma_n\|\boldsymbol{a}_n\|\leq\|\boldsymbol{a}_n\|/\beta.
\end{equation}
Hence, since $J_{\gamma_n\boldsymbol{A}}$ is nonexpansive,
\begin{align}
\label{e:broome-st2010-10-29f}
\|\barp_n-\boldsymbol{p}_n\|
&=\|J_{\gamma_n\boldsymbol{A}}\,\bary_n-
J_{\gamma_n\boldsymbol{A}}\boldsymbol{y}_n
-\boldsymbol{b}_n\|\nonumber\\
&\leq\|J_{\gamma_n \boldsymbol{A}}\,\bary_n-J_{\gamma_n
\boldsymbol{A}}\boldsymbol{y}_n\|+\|\boldsymbol{b}_n\|\nonumber\\
&\leq\|\bary_n-\boldsymbol{y}_n\|+\|\boldsymbol{b}_n\|\nonumber\\
&\leq\|\boldsymbol{a}_n\|/\beta+\|\boldsymbol{b}_n\|.
\end{align}
In turn, we get
\begin{align}
\label{e:broome-st2010-10-29d}
\|\barq_n-\boldsymbol{q}_n\|
&=\|\barp_n-\gamma_n\boldsymbol{B}\barp_n-\boldsymbol{p}_n
+\gamma_n(\boldsymbol{B}\boldsymbol{p}_n
+\boldsymbol{c}_n)\|\nonumber\\
&\leq\|\barp_n-\boldsymbol{p}_n\|+\gamma_n
\|\boldsymbol{B}\barp_n-\boldsymbol{B}\boldsymbol{p}_n\|
+\gamma_n\|\boldsymbol{c}_n\|
\nonumber\\
&\leq(1+\gamma_n\beta)\|\barp_n-\boldsymbol{p}_n\|
+\gamma_n\|\boldsymbol{c}_n\|\nonumber\\
&\leq2(\|\boldsymbol{a}_n\|/\beta+\|\boldsymbol{b}_n\|)
+\|\boldsymbol{c}_n\|/\beta.
\end{align}
Combining \eqref{e:broome-st2010-10-29a}, 
\eqref{e:broome-st2010-10-29c}, and \eqref{e:broome-st2010-10-29d} 
yields
$\|\boldsymbol{e}_n\|\leq\|\bary_n-\boldsymbol{y}_n\|+
\|\barq_n-\boldsymbol{q}_n\|\leq
3\|\boldsymbol{a}_n\|/\beta+2\|\boldsymbol{b}_n\|
+\|\boldsymbol{c}_n\|/\beta$
and, in view of \eqref{e:broome-st2010-10-27g}, it follows that
\begin{equation}
\label{e:broome-st2010-10-29l}
\sum_{k\in\NN}\|\boldsymbol{e}_k\|<\pinf.
\end{equation}
Furthermore, \eqref{e:rio2010-10-11u},
\eqref{e:broome-st2010-10-29a}, and \eqref{e:rio2010-10-11y}
imply that 
\begin{equation}
\label{e:broome-st2010-10-29e}
\|\boldsymbol{x}_{n+1}-\boldsymbol{x}\|
=\|\boldsymbol{x}_n-\boldsymbol{y}_n+
\boldsymbol{q}_n-\boldsymbol{x}\|
\leq\|\boldsymbol{x}_n-\bary_n+\barq_n-\boldsymbol{x}\|
+\|\boldsymbol{e}_n\|
\leq\|\boldsymbol{x}_n-\boldsymbol{x}\|+\|\boldsymbol{e}_n\|.
\end{equation}
Thus, it follows from \eqref{e:broome-st2010-10-29l} and 
Lemma~\ref{l:1} that $(\boldsymbol{x}_k)_{k\in\NN}$ is bounded, 
and we deduce from \eqref{e:broome-st2010-10-29a} that, since the 
operators $\boldsymbol{B}$ and
$(J_{\gamma_k\boldsymbol{A}})_{k\in\NN}$ are Lipschitzian, 
$(\bary_k)_{k\in\NN}$, $(\barp_k)_{k\in\NN}$,
and $(\barq_k)_{k\in\NN}$ are bounded.
Consequently, $\mu=\sup_{k\in\NN}\|\boldsymbol{x}_k-\bary_k+\barq_k
-\boldsymbol{x}\|<\pinf$ and, using \eqref{e:rio2010-10-11u}, 
\eqref{e:broome-st2010-10-29a}, 
and \eqref{e:rio2010-10-11y}, we obtain
\begin{align}
\label{e:broome-st2010-10-29m}
\|\boldsymbol{x}_{n+1}-\boldsymbol{x}\|^2
&=\|\boldsymbol{x}_n-\boldsymbol{y}_n+\boldsymbol{q}_n
-\boldsymbol{x}\|^2\nonumber\\
&=\|\boldsymbol{x}_n-\bary_n+\barq_n-\boldsymbol{x}
+\boldsymbol{e}_n\|^2\nonumber\\
&=\|\boldsymbol{x}_n-\bary_n+\barq_n-\boldsymbol{x}\|^2
+2\scal{\boldsymbol{x}_n-\bary_n+\barq_n-\boldsymbol{x}}
{\boldsymbol{e}_n}
+\|\boldsymbol{e}_n\|^2\nonumber\\
&\leq\|\boldsymbol{x}_n-\boldsymbol{x}\|^2
-\varepsilon^2\|\boldsymbol{x}_n-\barp_n\|^2
+\varepsilon_n,
\quad\text{where}\quad
\varepsilon_n=2\mu\|\boldsymbol{e}_n\|+\|\boldsymbol{e}_n\|^2.
\end{align}

\ref{t:broome-st2010-10-27i}:
It follows from \eqref{e:broome-st2010-10-29l}, 
\eqref{e:broome-st2010-10-29m}, and Lemma~\ref{l:1} that
\begin{equation}
\label{e:nyc2010-10-30a}
\sum_{n\in\NN}\|\boldsymbol{x}_n-\barp_n\|^2<\pinf.
\end{equation}
Hence, since \eqref{e:broome-st2010-10-27g} 
and \eqref{e:broome-st2010-10-29f} imply that
$\sum_{n\in\NN}\|\barp_n-\boldsymbol{p}_n\|<\pinf$, we have
$\sum_{n\in\NN}\|\barp_n-\boldsymbol{p}_n\|^2<\pinf$.
We therefore infer that 
$\sum_{n\in\NN}\|\boldsymbol{x}_n-\boldsymbol{p}_n\|^2<\pinf$.
Furthermore, since \eqref{e:broome-st2010-10-29a} yields
\begin{align}
\label{e:2010-11-09a}
(\forall n\in\NN)\quad\|\boldsymbol{y}_n-\boldsymbol{q}_n\|^2
&=\|\bary_n-\barq_n+\boldsymbol{e}_n\|^2\nonumber\\
&=\|\boldsymbol{x}_n-\barp_n-\gamma_n(\boldsymbol{B}\boldsymbol{x}_n
-\boldsymbol{B}\barp_n)+\boldsymbol{e}_n\|^2\nonumber\\
&\leq 3\big(\|\boldsymbol{x}_n-\barp_n\|^2
+\gamma_n^2\beta^2\,
\|\boldsymbol{x}_n-\barp_n\|^2+\|\boldsymbol{e}_n\|^2\big)\nonumber\\
&\leq 6\|\boldsymbol{x}_n-\barp_n\|^2+3\|\boldsymbol{e}_n\|^2,
\end{align}
we derive from \eqref{e:broome-st2010-10-29l} that
$\sum_{n\in\NN}\|\boldsymbol{y}_n-\boldsymbol{q}_n\|^2<\pinf$.

\ref{t:broome-st2010-10-27ii}:
It follows from \eqref{e:nyc2010-10-30a}, the Lipschitz 
continuity of $\boldsymbol{B}$, and \eqref{e:broome-st2010-10-29a}
that 
\begin{equation}
\label{e:rio2010-10-12e}
\boldsymbol{B}\barp_n-\boldsymbol{B}\boldsymbol{x}_n\to 0
\quad\text{and}\quad\boldsymbol{u}_n\to 0. 
\end{equation}
Now, let $\boldsymbol{w}$ be a weak sequential cluster point 
of $(\boldsymbol{x}_n)_{n\in\NN}$, say 
$\boldsymbol{x}_{k_n}\weakly\boldsymbol{w}$. 
It follows from \eqref{e:rio2010-10-12c} that
$(\barp_{k_n},\boldsymbol{u}_{k_n})_{n\in\NN}$ 
lies in $\gr(\boldsymbol{A}+\boldsymbol{B})$, and from  
\eqref{e:nyc2010-10-30a} and \eqref{e:rio2010-10-12e} that 
\begin{equation}
\label{e:rio2010-10-12a}
\barp_{k_n}\weakly\boldsymbol{w}\quad\text{and}\quad 
\boldsymbol{u}_{k_n}\to\boldsymbol{0}.
\end{equation}
Since $\boldsymbol{B}\colon\HHH\to\HHH$ is monotone and continuous, 
it is maximally monotone \cite[Example~20.29]{Livre1}. 
Furthermore, since $\dom \boldsymbol{B}=\HHH$, 
$\boldsymbol{A}+\boldsymbol{B}$ is maximally monotone 
\cite[Corollary~24.4(i)]{Livre1} and its graph is therefore 
sequentially closed in 
$\HHH^{\text{weak}}\times\HHH^{\text{strong}}$
\cite[Proposition~20.33(ii)]{Livre1}. Therefore, 
$(\boldsymbol{w},\boldsymbol{0})\in\gr(\boldsymbol{A}
+\boldsymbol{B})$. 
Using \eqref{e:broome-st2010-10-29m},
\eqref{e:broome-st2010-10-29l},
and Lemma~\ref{l:2}, we conclude that there exists 
$\boldsymbol{\overline{x}}\in\zer(\boldsymbol{A}+\boldsymbol{B})$ 
such that $\boldsymbol{x}_n\weakly\boldsymbol{\overline{x}}$.
Finally, in view of \ref{t:broome-st2010-10-27i}, 
$\boldsymbol{p}_n\weakly\boldsymbol{\overline{x}}$.

\ref{t:broome-st2010-10-27iiia}:
As shown in \ref{t:broome-st2010-10-27ii}, 
$\boldsymbol{p}_n\weakly\boldsymbol{\overline{x}}$. 
In turn, it follows from \eqref{e:broome-st2010-10-27g} that
$\barp_n=\boldsymbol{p}_n+\boldsymbol{b}_n\weakly
\boldsymbol{\overline{x}}$.
Moreover, \eqref{e:rio2010-10-12e} yields 
$\boldsymbol{u}_n\to\boldsymbol{0}$ and \eqref{e:rio2010-10-12c} 
yields $(\forall n\in\NN)$ 
$(\barp_n,\boldsymbol{u}_n)\in\gr(\boldsymbol{A}+\boldsymbol{B})$.
Altogether, Definition~\ref{d:demir} implies that
$\barp_n\to\boldsymbol{\overline{x}}$ and, therefore, that
$\boldsymbol{p}_n=\barp_n-\boldsymbol{b}_n\to\boldsymbol
{\overline{x}}$.
Finally, it results from \ref{t:broome-st2010-10-27i} that
$\boldsymbol{x}_n\to\boldsymbol{\overline{x}}$.

\ref{t:broome-st2010-10-27iiib}$\Rightarrow$%
\ref{t:broome-st2010-10-27iiia}: The assumptions imply that 
$\boldsymbol{A}+\boldsymbol{B}$ is uniformly monotone at 
$\boldsymbol{\overline{x}}$. Hence, the result follows 
from Lemma~\ref{l:2009-09-20}\ref{l:2009-09-20i}.

\ref{t:broome-st2010-10-27iiic}:
It follows from \eqref{e:broome-st2010-10-29m}, 
\eqref{e:broome-st2010-10-29l}, \ref{t:broome-st2010-10-27ii}, and 
\cite[Proposition~3.10]{Else01} that 
$\boldsymbol{x}_n\to\boldsymbol{\overline{x}}$. In turn, 
\ref{t:broome-st2010-10-27i} yields
$\boldsymbol{p}_n\to\boldsymbol{\overline{x}}$. 
\end{proof}

\begin{remark}
\label{r:nyc2010-10-30}
The sequence $(\boldsymbol{a}_n)_{n\in\NN}$,
$(\boldsymbol{b}_n)_{n\in\NN}$, and 
$(\boldsymbol{c}_n)_{n\in\NN}$ in \eqref{e:rio2010-10-11u}
model errors in the implementation of the operators.
In the error-free setting, the weak convergence of
$(\boldsymbol{x}_n)_{n\in\NN}$ to a zero of
$\boldsymbol{A}+\boldsymbol{B}$ in 
Theorem~\ref{t:broome-st2010-10-27}\ref{t:broome-st2010-10-27ii}
follows from \cite[Theorem~3.4(b)]{Tsen00}.
\end{remark}

\subsection{The monotone{\small$+$}skew model}

Let us start with some elementary facts about the operators
$\boldsymbol{M}$ and $\boldsymbol{S}$ appearing in 
Problem~\ref{prob:2}.

\begin{proposition}
\label{p:1}
Consider the setting of Problem~\ref{prob:1} and 
Problem~\ref{prob:2}. Then the following hold.
\begin{enumerate}
\item 
\label{p:1i----} 
$\boldsymbol{M}$ is maximally monotone.
\item 
\label{p:1i---} 
$\boldsymbol{S}\in\BL(\KKK)$, $\boldsymbol{S}^*=-\boldsymbol{S}$,
and $\|\boldsymbol{S}\|=\|L\|$.
\item 
\label{p:1i--} 
$\boldsymbol{M}+\boldsymbol{S}$ is maximally monotone.
\item 
\label{p:1i+} 
$(\forall\gamma\in\RPP)(\forall x\in\HH)(\forall v\in\GG)$ 
$J_{\gamma\boldsymbol{M}}(x,v)=\big(J_{\gamma A}(x+\gamma z)
\,,\,J_{\gamma B^{-1}}(v-\gamma r)\big)$.
\item 
\label{p:1i++} 
$(\forall\gamma\in\RPP)(\forall x\in\HH)(\forall v\in\GG)$ 
\[
J_{\gamma\boldsymbol{S}}(x,v)=\big((\Id+\,\gamma^2L^*L)^{-1}
(x-\gamma L^*v)\,,\,(\Id+\,\gamma^2LL^*)^{-1}(v+\gamma Lx)\big).
\]
\end{enumerate}
\end{proposition}
\begin{proof} 
\ref{p:1i----}: 
Since $A$ and $B$ are maximally monotone, it follows from
\cite[Propositions~20.22 and~20.23]{Livre1} that $A\times B^{-1}$ 
is likewise. In turn, $\boldsymbol{M}$ is maximally monotone.

\ref{p:1i---}:
The first two assertions are clear. 
Now let $(x,v)\in\KKK$. Then
$\|\boldsymbol{S}(x,v)\|^2=\|(L^*v,-Lx)\|^2
=\|L^*v\|^2+\|Lx\|^2\leq\|L\|^2(\|v\|^2+\|x\|^2)
=\|L\|^2\|(x,v)\|^2$. Thus, $\|\boldsymbol{S}\|\leq\|L\|$.
Conversely, $\|x\|\leq 1$ $\Rightarrow$ $\|(x,0)\|\leq 1$
$\Rightarrow$ $\|Lx\|=\|\boldsymbol{S}(x,0)\|\leq\|\boldsymbol{S}\|$.
Hence $\|L\|\leq\|\boldsymbol{S}\|$.

\ref{p:1i--}:
By \ref{p:1i----}, $\boldsymbol{M}$ is maximally monotone. On the 
other hand, it follows from \ref{p:1i---} that $\boldsymbol{S}$ 
is monotone and continuous, hence maximally monotone 
\cite[Example~20.29]{Livre1}. Altogether, since 
$\dom\boldsymbol{S}=\KKK$, it follows from 
\cite[Corollary 24.4]{Livre1} that 
$\boldsymbol{M}+\boldsymbol{S}$ is maximally monotone.

\ref{p:1i+}: This follows from \cite[Proposition~23.16]{Livre1}.

\ref{p:1i++}: 
Let $(x,v)\in\KKK$ and set $(p,q)=J_{\gamma\boldsymbol{S}}(x,v)$.
Then $(x,v)=(p,q)+\gamma\boldsymbol{S}(p,q)$ and hence
$x=p+\gamma L^*q$ and $v=q-\gamma Lp$. Hence, 
$Lx=Lp+\gamma LL^*q$ and $L^*v=L^*q-\gamma L^*Lp$.
Thus, $x=p+\gamma L^*v+\gamma^2 L^*Lp$ and therefore
$p=(\Id+\,\gamma^2L^*L)^{-1}(x-\gamma L^*v)$. Likewise, 
$v=q-\gamma Lx+\gamma^2 LL^*q$, and therefore 
$q=(\Id+\,\gamma^2LL^*)^{-1}(v+\gamma Lx)$.
\end{proof} 

The next proposition makes the tight interplay between 
Problem~\ref{prob:1} and Problem~\ref{prob:2} explicit. 
An alternate proof of the equivalence
\ref{p:2ii}$\Leftrightarrow$\ref{p:2iii}$\Leftrightarrow$\ref{p:2iv}
can be found in \cite{Penn00} (see also 
\cite{Atto96,Ecks99,Merc80,Robi99} for partial results); 
we provide a direct argument for completeness.

\begin{proposition}
\label{p:2}
Consider the setting of Problem~\ref{prob:1} and 
Problem~\ref{prob:2}. Then 
\begin{enumerate}
\item 
\label{p:2i-} 
$\zer(\boldsymbol{M}+\boldsymbol{S})$ is a closed convex
subset of ${\mathcal P}\times{\mathcal D}$.
\end{enumerate}
Furthermore, the following are equivalent. 
\begin{enumerate}
\setcounter{enumi}{1}
\item 
\label{p:2i} 
$z\in\ran(A+L^*\circ B\circ (L\cdot-r))$.
\item 
\label{p:2ii}
${\mathcal P}\neq\emp$.
\item 
\label{p:2iii} 
$\zer(\boldsymbol{M}+\boldsymbol{S})\neq\emp$.
\item 
\label{p:2iv}
${\mathcal D}\neq\emp$.
\item 
\label{p:2v}
$-r\in\ran (-L\circ A^{-1}\circ (z-L^*\cdot)+B^{-1})$.
\end{enumerate}
\end{proposition}
\begin{proof} 
The equivalences \ref{p:2i}$\Leftrightarrow$\ref{p:2ii} and
\ref{p:2iv}$\Leftrightarrow$\ref{p:2v} are clear.
Now let $(x,v)\in\KKK$.

\ref{p:2i-}:
We derive from from \eqref{e:nyc2010-10-31z} that
$(x,v)\in\zer(\boldsymbol{M}+\boldsymbol{S})$
$\Leftrightarrow$ $0\in -z+Ax+L^*v$ and 
$0\in r+B^{-1}v-Lx$ $\Leftrightarrow$
($z-L^*v\in Ax$ and $Lx-r\in B^{-1}v$)
$\Leftrightarrow$ ($z-L^*v\in Ax$ and $v\in B(Lx-r)$)
$\Rightarrow$ ($z-L^*v\in Ax$ and $L^*v\in L^*(B(Lx-r))$)
$\Rightarrow$ $z\in Ax+L^*(B(Lx-r))$ $\Leftrightarrow$
$x\in{\mathcal P}$.
Similarly,
($z-L^*v\in Ax$ and $Lx-r\in B^{-1}v$)
$\Leftrightarrow$ ($x\in A^{-1}(z-L^*v)$ and $r-Lx\in -B^{-1}v$)
$\Rightarrow$ ($Lx\in L(A^{-1}(z-L^*v))$ and $r-Lx\in -B^{-1}v$)
$\Rightarrow$ $r\in L(A^{-1}(z-L^*v))-B^{-1}v$ $\Leftrightarrow$
$v\in{\mathcal D}$. Finally, since $\boldsymbol{M}+\boldsymbol{S}$ 
is maximally monotone by Proposition~\ref{p:1}\ref{p:1i--}, 
$\zer(\boldsymbol{M}+\boldsymbol{S})$ is closed and convex
\cite[Proposition~23.39]{Livre1}.

\ref{p:2ii}$\Rightarrow$\ref{p:2iii}: 
In view of \eqref{e:nyc2010-10-31z}, 
$x\in{\mathcal P}\Leftrightarrow
z\in Ax+L^*(B(Lx-r))\Leftrightarrow (\exi w\in\GG)\; 
\big(z-L^*w\in Ax\;\text{and}\;w\in B(Lx-r)\big)
\Leftrightarrow
\big((\exi w\in\GG)\;z\in Ax+L^*w\;\text{and}\;-r\in B^{-1}w-Lx\big)
\Leftrightarrow (\exi w\in\GG)\;(x,w)
\in\zer(\boldsymbol{M}+\boldsymbol{S})$.

\ref{p:2iii}$\Rightarrow$\ref{p:2ii} and 
\ref{p:2iii}$\Rightarrow$\ref{p:2iv}: 
These follow from \ref{p:2i-}.

\ref{p:2iv}$\Rightarrow$\ref{p:2iii}: 
$v\in{\mathcal D}
\Leftrightarrow
r\in LA^{-1}(z-L^*v)-B^{-1}v
\Leftrightarrow(\exi y\in\HH)\; 
(y\in A^{-1}(z-L^*v)\;\text{and}\;
r\in Ly$ $-B^{-1}v)
\Leftrightarrow(\exi y\in\HH)\; 
(0\in -z+Ay+L^*v\;\text{and}\;
0\in r+B^{-1}v-Ly)
\Leftrightarrow(\exi y\in\HH)\;(y,v)
\in\zer(\boldsymbol{M}+\boldsymbol{S})$.
\end{proof}

\begin{remark}
\label{r:7}
Suppose that $z\in\ran(A+L^*B(L\cdot-r))$. Then
Proposition~\ref{p:2} assert that solutions to
\eqref{e:primal} and \eqref{e:dual} can be found
as zeros of $\boldsymbol{M}+\boldsymbol{S}$. 
In principle, this can be achieved via the 
Douglas-Rachford algorithm applied to \eqref{e:trial}: 
let $(\boldsymbol{a}_n)_{n\in\NN}$ and 
$(\boldsymbol{b}_n)_{n\in\NN}$ be sequences in $\KKK$, let 
$(\lambda_n)_{n\in\NN}$ be a sequence in $\left]0,2\right[$ 
such that $\boldsymbol{b}_n\weakly\boldsymbol{0}$,
$\sum_{n\in\NN}\lambda_n(\|\boldsymbol{a}_n\|+ 
\|\boldsymbol{b}_n\|)<\pinf$, and 
$\sum_{n\in\NN}\lambda_n(2-\lambda_n)=\pinf$,
let $\boldsymbol{y}_0\in\KKK$, let 
$\gamma\in\RPP$, and set
\begin{equation}
\label{e:7ans}
(\forall n\in\NN)\quad
\begin{array}{l}
\left\lfloor
\begin{array}{l}
\boldsymbol{x}_n=J_{\gamma \boldsymbol{S}}\,\boldsymbol{y}_n
+\boldsymbol{b}_n\\
\boldsymbol{y}_{n+1}=\boldsymbol{y}_n+\lambda_n
\big(J_{\gamma \boldsymbol{M}}(2\boldsymbol{x}_n-\boldsymbol{y}_n)
+\boldsymbol{a}_n-\boldsymbol{x}_n\big).
\end{array}
\right.
\end{array}
\end{equation}
Then it follows from 
Proposition~\ref{p:1}\ref{p:1i----}--\ref{p:1i--} and 
\cite[Theorem~2.1(i)(c)]{Joca09} that 
$(\boldsymbol{x}_n)_{n\in\NN}$ converges weakly to a point in
$\zer(\boldsymbol{M}+\boldsymbol{S})$. 
Now set $(\forall n\in\NN)$
$\boldsymbol{x}_n=(x_n,v_n)$,
$\boldsymbol{y}_n=(y_{1,n},y_{2,n})$,
$\boldsymbol{a}_n=(a_{1,n},a_{2,n})$,
and $\boldsymbol{b}_n=(b_{1,n},b_{2,n})$.
Then, using Proposition~\ref{p:1}\ref{p:1i+}\&\ref{p:1i++},
\eqref{e:7ans} becomes
\begin{equation}
\label{e:main3}
(\forall n\in \NN)\quad
\begin{array}{l}
\left\lfloor
\begin{array}{l}
x_n=(\Id+\,\gamma^2L^*L)^{-1}(y_{1,n}-\gamma L^*y_{2,n})+b_{1,n}\\
v_n=(\Id+\,\gamma^2LL^*)^{-1}(y_{2,n}+\gamma Ly_{1,n})+b_{2,n}\\
y_{1,n+1}=y_{1,n}+\lambda_n\big(J_{\gamma A}(2x_n-y_{1,n}
+\gamma z)+a_{1,n}-x_n\big)\\
y_{2,n+1}=y_{2,n}+\lambda_n
\big(J_{\gamma B^{-1}}(2v_n-y_{2,n}-\gamma r)+a_{2,n}-v_n\big).
\end{array}
\right.\\[2mm]
\end{array}
\end{equation}
Moreover, $(x_n)_{n\in\NN}$ converges weakly to a solution 
$\overline{x}$ to \eqref{e:primal} and $(v_n)_{n\in\NN}$ to a 
solution $\overline{v}$ to \eqref{e:dual} such that 
$z-L^*\overline{v}\in A\overline{x}$ and 
$\overline{v}\in B(L\overline{x}-r)$.
However, a practical limitation of \eqref{e:main3} is that it 
necessitates the inversion of two operators at each iteration,
which may be quite demanding numerically.
\end{remark}

\begin{remark}
\label{r:remhalf}
It follows from \eqref{e:broome-st2010-10-29e} that the error-free 
version of the forward-backward-forward algorithm 
\eqref{e:rio2010-10-11u} is Fej\'er-monotone with respect to
$\zer(\boldsymbol{A}+\boldsymbol{B})$, i.e., for every $n\in\NN$ and
every $\boldsymbol{x}\in\zer(\boldsymbol{A}+\boldsymbol{B})$,
$\|\boldsymbol{x}_{n+1}-\boldsymbol{x}\|\leq
\|\boldsymbol{x}_n-\boldsymbol{x}\|$. Now let $n\in\NN$. 
Then it follows from \cite[Section~2]{Moor01} that there exist 
$\lambda_n\in [0,2]$ and a closed affine halfspace 
$\boldsymbol{H}_n\subset\HHH$ containing 
$\zer(\boldsymbol{A}+\boldsymbol{B})$ such that 
\begin{equation}
\label{e:2010-11-20a}
\boldsymbol{x}_{n+1}=\boldsymbol{x}_n+\lambda_n(P_{\boldsymbol{H}_n}
\boldsymbol{x}_n-\boldsymbol{x}_n).
\end{equation}
In the setting of Problem~\ref{prob:2}, $\boldsymbol{H}_n$ and 
$\lambda_n$ can be determined easily. To see this, consider 
Theorem~\ref{t:broome-st2010-10-27} with $\HHH=\KKK$,
$\boldsymbol{A}=\boldsymbol{M}$, and $\boldsymbol{B}=\boldsymbol{S}$.
Let $\boldsymbol{\overline{x}}\in\zer(\boldsymbol{M}+\boldsymbol{S})$
and suppose that $\boldsymbol{q}_n\neq\boldsymbol{y}_n$ (otherwise, 
we trivially have $\boldsymbol{H}_n=\KKK$).
In view of \eqref{e:rio2010-10-11u},
$\boldsymbol{y}_n-\boldsymbol{p}_n\in\gamma_n\boldsymbol{M}
\boldsymbol{p}_n$ and 
$-\gamma_n\boldsymbol{S}\boldsymbol{\overline{x}}\in\gamma_n
\boldsymbol{M}\boldsymbol{\overline{x}}$.
Hence, using the monotonicity of $\gamma_n\boldsymbol{M}$ and
Proposition~\ref{p:1}\ref{p:1i---}, we get
$0\leq\scal{\boldsymbol{p}_n-\boldsymbol{\overline{x}}}
{\boldsymbol{y}_n-\boldsymbol{p}_n+\gamma_n\boldsymbol{S}
\boldsymbol{\overline{x}}}
=\scal{\boldsymbol{p}_n}{\boldsymbol{y}_n-\boldsymbol{p}_n}-
\scal{\boldsymbol{x}}{\boldsymbol{y}_n-\boldsymbol{p}_n}+
\gamma_n\scal{\boldsymbol{S}^*\boldsymbol{p}_n}
{\boldsymbol{\overline{x}}}
=\scal{\boldsymbol{p}_n}{\boldsymbol{y}_n-\boldsymbol{p}_n}-
\scal{\boldsymbol{\overline{x}}}{\boldsymbol{y}_n-\boldsymbol{p}_n+
\gamma_n\boldsymbol{S}\boldsymbol{p}_n}$.
Therefore, we deduce from \eqref{e:rio2010-10-11u} that   
$\scal{\boldsymbol{\overline{x}}}{\boldsymbol{y}_n-\boldsymbol{q}_n}
\leq\scal{\boldsymbol{p}_n}{\boldsymbol{y}_n-\boldsymbol{p}_n}
=\scal{\boldsymbol{p}_n}{\boldsymbol{y}_n-\boldsymbol{q}_n}$.
Now set
\begin{equation}
\label{e:2010-11-20b}
\boldsymbol{H}_n=\menge{\boldsymbol{x}\in\KKK}
{\scal{\boldsymbol{x}}{\boldsymbol{y}_n-\boldsymbol{q}_n}\leq 
\scal{\boldsymbol{p}_n}{\boldsymbol{y}_n-\boldsymbol{q}_n}} 
\quad\text{and}\quad
\lambda_n=1+\gamma_n^2\frac{\|\boldsymbol{S}(\boldsymbol{p}_n-
\boldsymbol{x}_n)\|^2}{\|\boldsymbol{p}_n-\boldsymbol{x}_n\|^2}.
\end{equation}
Then $\zer(\boldsymbol{M}+\boldsymbol{S})\subset\boldsymbol{H}_n$
and $\lambda_n\leq 1+\gamma_n^2\|\boldsymbol{S}\|^2<2$.
Altogether, it follows from \eqref{e:rio2010-10-11u} and the
skew-adjointness of $\boldsymbol{S}$ that
\begin{align}
\label{e:xn+1}
\boldsymbol{x}_n+\lambda_n(\boldsymbol{P}_{\boldsymbol{H}_n}
\boldsymbol{x}_n-\boldsymbol{x}_n)
&=\boldsymbol{x}_n+\lambda_n\bigg(\frac{\scal{\boldsymbol{p}_n-
\boldsymbol{x}_n}{\boldsymbol{y}_n-
\boldsymbol{q}_n}}{\|\boldsymbol{y}_n-
\boldsymbol{q}_n\|^2}\bigg)(\boldsymbol{y}_n-
\boldsymbol{q}_n)\nonumber\\
&=\boldsymbol{x}_n+\lambda_n\bigg(\frac{\scal{\boldsymbol{p}_n-
\boldsymbol{x}_n}{\boldsymbol{x}_n-
\boldsymbol{p}_n+\gamma_n\boldsymbol{S}(\boldsymbol{p}_n-
\boldsymbol{x}_n)}}{\|\boldsymbol{x}_n-
\boldsymbol{p}_n+\gamma_n\boldsymbol{S}(\boldsymbol{p}_n-
\boldsymbol{x}_n)\|^2}\bigg)(\boldsymbol{y}_n-
\boldsymbol{q}_n)\nonumber\\
&=\boldsymbol{x}_n+\lambda_n\bigg(\frac{\|\boldsymbol{x}_n-
\boldsymbol{p}_n\|^2}{\|\boldsymbol{x}_n-
\boldsymbol{p}_n\|^2+\gamma_n^2\|\boldsymbol{S}(\boldsymbol{p}_n-
\boldsymbol{x}_n)\|^2}\bigg)(\boldsymbol{q}_n-
\boldsymbol{y}_n)\nonumber\\
&=\boldsymbol{x}_n-\boldsymbol{y}_n+
\boldsymbol{q}_n
=\boldsymbol{x}_{n+1}.
\end{align}
Thus, the updating rule of algorithm of 
Theorem~\ref{t:broome-st2010-10-27} applied to $\boldsymbol{M}$ 
and $\boldsymbol{S}$ is given by 
\eqref{e:2010-11-20a}--\eqref{e:2010-11-20b}.
In turn, using results from \cite{Moor01}, this iteration process
can easily be modified to become strongly convergent.
\end{remark}

\section{Main results}
\label{sec:3}

The main result of the paper can now be presented. It consists of
an application of Theorem~\ref{t:broome-st2010-10-27} to find 
solutions to Problem~\ref{prob:2}, and thus obtain solutions to
Problem~\ref{prob:1}. The resulting algorithm employs the operators
$A$, $B$, and $L$ separately. Moreover, the operators $A$ and $B$
can be activated in parallel and all the steps involving $L$
are explicit.

\begin{theorem}
\label{t:nyc2010-10-31}
In Problem~\ref{prob:1}, suppose that $L\neq 0$ and that
$z\in\ran\big(A+L^*\circ B\circ (L\cdot-r)\big)$.
Let $(a_{1,n})_{n\in\NN}$, $(b_{1,n})_{n\in\NN}$, and
$(c_{1,n})_{n\in\NN}$ be absolutely summable sequences in $\HH$,
and let $(a_{2,n})_{n\in\NN}$, $(b_{2,n})_{n\in\NN}$, and
$(c_{2,n})_{n\in\NN}$ be absolutely summable sequences in $\GG$.
Furthermore, let $x_0\in\HH$, let $v_0\in\GG$, let 
$\varepsilon\in\left]0,1/(\|L\|+1)\right[$, 
let $(\gamma_n)_{n\in\NN}$ be a sequence in 
$[\varepsilon,(1-\varepsilon)/\|L\|\,]$, and set
\begin{equation}
\label{e:nypl2010-10-27b}
(\forall n\in\NN)\quad 
\begin{array}{l}
\left\lfloor
\begin{array}{l}
y_{1,n}=x_n-\gamma_n(L^*v_n+a_{1,n})\\
y_{2,n}=v_n+\gamma_n(Lx_n+a_{2,n})\\
p_{1,n}=J_{\gamma_n A}(y_{1,n}+\gamma_nz)+b_{1,n}\\
p_{2,n}=J_{\gamma_n B^{-1}}(y_{2,n}-\gamma_nr)+b_{2,n}\\
q_{1,n}=p_{1,n}-\gamma_n(L^*p_{2,n}+c_{1,n})\\
q_{2,n}=p_{2,n}+\gamma_n(Lp_{1,n}+c_{2,n})\\
x_{n+1}=x_n-y_{1,n}+q_{1,n}\\
v_{n+1}=v_n-y_{2,n}+q_{2,n}.
\end{array}
\right.\\
\end{array}
\end{equation}
Then the following hold for some solution $\overline{x}$ to 
\eqref{e:primal} and some solution $\overline{v}$ to 
\eqref{e:dual} such that 
$z-L^*\overline{v}\in A\overline{x}$ and 
$\overline{v}\in B(L\overline{x}-r)$.
\begin{enumerate}
\item
\label{t:nyc2010-10-31i-}
$x_n-p_{1,n}\to 0$ and $v_n-p_{2,n}\to 0$. 
\item
\label{t:nyc2010-10-31i}
$x_n\weakly\overline{x}$, $p_{1,n}\weakly\overline{x}$,  
$v_n\weakly\overline{v}$, and $p_{2,n}\weakly\overline{v}$.
\item
\label{t:nyc2010-10-31iii} 
Suppose that $A$ is uniformly monotone at $\overline{x}$.
Then $x_n\to\overline{x}$ and $p_{1,n}\to\overline{x}$.
\item
\label{t:nyc2010-10-31iv}
Suppose that $B^{-1}$ is uniformly monotone at $\overline{v}$.
Then $v_n\to\overline{v}$ and $p_{2,n}\to\overline{v}$.
\end{enumerate}
\end{theorem}
\begin{proof}
Consider the setting of Problem~\ref{prob:2}.
As seen in Proposition~\ref{p:1}, $\boldsymbol{M}$ 
is maximally monotone, and $\boldsymbol{S}\in\BL(\KKK)$ is 
monotone and Lipschitzian with constant $\|L\|$.
Moreover, Proposition~\ref{p:2} yields
\begin{equation}
\label{e:2010-11-10d}
\emp\neq\zer(\boldsymbol{M}+\boldsymbol{S})\subset
{\mathcal P}\times{\mathcal D}.
\end{equation}
Now set 
\begin{equation}
\label{e:nyc2010-10-31y'}
(\forall n\in\NN)\quad
\begin{cases}
\boldsymbol{x}_n=(x_n,v_n)\\
\boldsymbol{y}_n=(y_{1,n},y_{2,n})\\
\boldsymbol{p}_n=(p_{1,n},p_{2,n})\\
\barp_n=(\widetilde{p}_{1,n},\widetilde{p}_{2,n})\\
\boldsymbol{q}_n=(q_{1,n},q_{2,n}).
\end{cases}
\quad\text{and}\qquad
\begin{cases}
\boldsymbol{a}_n=(a_{1,n},a_{2,n})\\
\boldsymbol{b}_n=(b_{1,n},b_{2,n})\\
\boldsymbol{c}_n=(c_{1,n},c_{2,n}).
\end{cases}
\end{equation}
Then, using \eqref{e:nyc2010-10-31z} and 
Proposition~\ref{p:1}\ref{p:1i+},
\eqref{e:nypl2010-10-27b} we can written as 
\eqref{e:rio2010-10-11u} in $\KKK$. Moreover, our assumptions
imply that \eqref{e:broome-st2010-10-27g} is satisfied.
Hence, using \eqref{e:broome-st2010-10-29f}, we obtain
\begin{equation}
\label{e:2010-11-11j}
p_{1,n}-\widetilde{p}_{1,n}\to 0\quad\text{and}\quad
p_{2,n}-\widetilde{p}_{2,n}\to 0.
\end{equation}
Furthermore, we derive from \eqref{e:broome-st2010-10-29a} and 
\eqref{e:nyc2010-10-31z}
that 
\begin{equation}
\label{e:prim_dual}
(\forall n\in\NN)\quad
\begin{cases}
\gamma_n^{-1}(x_n-{\widetilde p}_{1,n})
+L^*{\widetilde p}_{2,n}-L^*v_n\in 
A{\widetilde p}_{1,n}-z+L^*{\widetilde p}_{2,n}\\
\gamma_n^{-1}(v_n-{\widetilde p}_{2,n})
-L{\widetilde p}_{1,n}+Lx_n\in B^{-1}{\widetilde p}_{2,n}
+r-L{\widetilde p}_{1,n}.
\end{cases}
\end{equation}
These observations allow us to establish the following.

\ref{t:nyc2010-10-31i-}\&\ref{t:nyc2010-10-31i}: 
These follows from Theorem~\ref{t:broome-st2010-10-27}%
\ref{t:broome-st2010-10-27i}\&\ref{t:broome-st2010-10-27ii}
applied to $\boldsymbol{M}$ and $\boldsymbol{S}$ in $\KKK$.

\ref{t:nyc2010-10-31iii}: 
Since $\overline{x}$ solves \eqref{e:primal}, 
there exist $u\in\HH$ and $v\in\GG$ such that
\begin{equation}
\label{e:solprim}
u\in A\overline{x},\quad v\in B(L\overline{x}-r),
\quad\text{and}\quad z=u+L^*v.
\end{equation}
Now let $n\in\NN$. We derive from \eqref{e:prim_dual} that
\begin{equation}
\gamma_n^{-1}(x_n-{\widetilde p}_{1,n})-L^*v_n+z\in 
A{\widetilde p}_{1,n}
\quad\text{and}\quad
\gamma_n^{-1}(v_n-{\widetilde p}_{2,n})+Lx_n-r\in 
B^{-1}{\widetilde p}_{2,n},
\end{equation}
which yields
\begin{equation}
\label{e:2010-11-11a}
\gamma_n^{-1}(x_n-{\widetilde p}_{1,n})-L^*v_n+z\in 
A{\widetilde p}_{1,n}
\quad\text{and}\quad
{\widetilde p}_{2,n}\in B\big(\gamma_n^{-1}(v_n-{\widetilde p}_{2,n})
+Lx_n-r\big).
\end{equation}
Now set
\begin{equation}
\label{e:2010-11-11b}
\alpha_n=\|x_n-{\widetilde p}_{1,n}\|\big(\varepsilon^{-1}
\|{\widetilde p}_{1,n}-\overline{x}\|
+\|L\|\,\|v_n-v\|\big)
\quad\text{and}\quad
\beta_n=\varepsilon^{-1}\|v_n-{\widetilde p}_{2,n}\|\,
\|{\widetilde p}_{2,n}-v\|.
\end{equation}
It follows from \eqref{e:solprim}, 
\eqref{e:2010-11-11a},
and the uniform monotonicity of $A$ that there exists 
an increasing function $\phi\colon\RP\to\RPX$ that vanishes only 
at $0$ such that
\begin{multline}
\label{e:Aunifmon}
\alpha_n+\scal{x_n-\overline{x}}{L^*v-L^*v_n}\\
\begin{aligned}[b]
&\geq\varepsilon^{-1}\|{\widetilde p}_{1,n}-\overline{x}\|\,
\|x_n-{\widetilde p}_{1,n}\|
+\scal{{\widetilde p}_{1,n}-x_n}{L^*v-L^*v_n}
+\scal{x_n-\overline{x}}{L^*v-L^*v_n}\\
&=\varepsilon^{-1}\|{\widetilde p}_{1,n}-\overline{x}\|\,
\|x_n-{\widetilde p}_{1,n}\|
+\scal{{\widetilde p}_{1,n}-\overline{x}}{L^*v-L^*v_n}\\
&\geq\scal{{\widetilde p}_{1,n}-\overline{x}}
{\gamma_n^{-1}(x_n-{\widetilde p}_{1,n})-L^*v_n+L^*v}\\
&=\scal{{\widetilde p}_{1,n}-\overline{x}}
{\gamma_n^{-1}(x_n-{\widetilde p}_{1,n})-L^*v_n+z-u}\\
&\geq\phi(\|{\widetilde p}_{1,n}-\overline{x}\|).
\end{aligned}
\end{multline}
On the other hand, since $B$ is monotone, 
\eqref{e:2010-11-11b}, 
\eqref{e:solprim}, and \eqref{e:2010-11-11a} yield
\begin{align}
\label{e:Bmon}
\beta_n+\scal{x_n-\overline{x}}{L^*{\widetilde p}_{2,n}-L^*v}
&\geq\scal{\gamma_n^{-1}(v_n-{\widetilde p}_{2,n})
+L(x_n-\overline{x})}{{\widetilde p}_{2,n}-v}\nonumber\\
&=\scal{(\gamma_n^{-1}(v_n-{\widetilde p}_{2,n})
+Lx_n-r)-(L\overline{x}-r)}{{\widetilde p}_{2,n}-v}\nonumber\\
&\geq 0.
\end{align}
Upon adding these two inequalities, we obtain
\begin{equation}
\label{e:2010-11-11c}
\alpha_n+\beta_n+\|x_n-\overline{x}\|\,\|L\|\,
\|{\widetilde p}_{2,n}-v_n\|
\geq\alpha_n+\beta_n+\scal{x_n-\overline{x}}
{L^*({\widetilde p}_{2,n}-v_n)}\geq
\phi(\|{\widetilde p}_{1,n}-\overline{x}\|).
\end{equation}
Hence, since \ref{t:nyc2010-10-31i}, \ref{t:nyc2010-10-31i-},
and \eqref{e:2010-11-11j} imply that the sequences
$(x_n)_{n\in\NN}$, $(v_n)_{n\in\NN}$,
$({\widetilde p}_{1,n})_{n\in\NN}$, and
$({\widetilde p}_{2,n})_{n\in\NN}$ are bounded, it 
follows from \eqref{e:2010-11-11b}, \eqref{e:2010-11-11j},
and \ref{t:nyc2010-10-31i-} that 
$\phi(\|{\widetilde p}_{1,n}-\overline{x}\|)\to 0$, 
from which we infer that ${\widetilde p}_{1,n}\to\overline{x}$ 
and, by \eqref{e:2010-11-11j}, that $p_{1,n}\to \overline{x}$. 
In turn, \ref{t:nyc2010-10-31i-} yields
$x_n\to \overline{x}$.

\ref{t:nyc2010-10-31iv}: Proceed as in 
\ref{t:nyc2010-10-31iii}, using the dual objects.
\end{proof}

\begin{remark}
\label{r:1}
Using a well-known resolvent identity, the computation of 
$p_{2,n}$ in \eqref{e:nypl2010-10-27b} can be 
performed in terms of the resolvent of $B$ via the identity
$J_{\gamma_nB^{-1}}y=y-\gamma_n J_{\gamma_n^{-1}B}(\gamma_n^{-1}y)$.
\end{remark}

\begin{remark}
Set $\boldsymbol{\mathcal {Z}}=
\menge{(x,v)\in{\mathcal P}\times{\mathcal D}}
{z-L^*{v}\in A{x}\;\:\text{and}\;\:{v}\in B(L{x}-r)}$.
Since Theorem~\ref{t:nyc2010-10-31} is an application of
Theorem~\ref{t:broome-st2010-10-27} in $\KKK$, we deduce from 
Remark~\ref{r:remhalf} that the updating process for $(x_n,v_n)$ 
in \eqref{e:nypl2010-10-27b} results from a relaxed projection
onto a closed affine halfspace $\boldsymbol{H}_n$ containing 
$\boldsymbol{\mathcal Z}$, namely 
\begin{equation}
\label{e:valparaiso2010-11-24a}
({x}_{n+1},{v}_{n+1})=({x}_{n},{v}_{n})+\lambda_n\big(
P_{\boldsymbol{H}_n}({x}_{n},{v}_{n})-({x}_{n},{v}_{n})\big),
\end{equation}
where
\begin{multline}
\label{e:halfsp2}
\boldsymbol{H}_n=\menge{(x,v)\in\KKK}
{\scal{x}{y_{1,n}-q_{1,n}}+
\scal{v}{y_{2,n}-q_{2,n}}\leq 
\scal{p_{1,n}}{y_{1,n}-q_{1,n}}
+\scal{p_{2,n}}{y_{2,n}-q_{2,n}}}\\
\text{and}\quad
\lambda_n=1+\gamma_n^2\frac{\|L(p_{1,n}-x_n)\|^2+
\|L^*(p_{2,n}-v_n)\|^2}{\|p_{1,n}-x_n\|^2+
\|p_{2,n}-v_n\|^2}.
\end{multline}
In the special case when $\GG=\HH$ and $L=\Id$, 
an analysis of such outer projection methods is provided in 
\cite{Svai08}.
\end{remark}

\begin{corollary}
\label{c:nyc2010-10-31-1}
Let $A_1\colon\HH\to 2^{\HH}$ and $A_2\colon\HH\to 2^{\HH}$ be 
maximally monotone operators such that $\zer(A_1+A_2)\neq\emp$.
Let $(b_{1,n})_{n\in\NN}$ and $(b_{2,n})_{n\in\NN}$ be absolutely 
summable sequences in $\HH$, let $x_0$ and $v_0$ be in $\HH$, 
let $\varepsilon\in\left]0,1/2\right[$, let $(\gamma_n)_{n\in\NN}$ 
be a sequence in $[\varepsilon,1-\varepsilon]$, and set
\begin{equation}
\label{e:nypl2010-10-27s}
(\forall n\in\NN)\quad 
\begin{array}{l}
\left\lfloor
\begin{array}{l}
p_{1,n}=J_{\gamma_n A_1}(x_n-\gamma_nv_n)+b_{1,n}\\
p_{2,n}=J_{\gamma_n A_2^{-1}}(v_n+\gamma_nx_n)+b_{2,n}\\
x_{n+1}=p_{1,n}+\gamma_n(v_n-p_{2,n})\\
v_{n+1}=p_{2,n}+\gamma_n(p_{1,n}-x_n).
\end{array}
\right.\\
\end{array}
\end{equation}
Then the following hold for some $\overline{x}\in\zer(A_1+A_2)$ 
and some
$\overline{v}\in\zer(-A_1^{-1}\circ(-\Id)+A_2^{-1})$ such that 
$-\overline{v}\in A_1\overline{x}$ and 
$\overline{v}\in A_2\overline{x}$.
\begin{enumerate}
\item
\label{c:nyc2010-10-31-ai}
$x_n\weakly\overline{x}$ and $v_n\weakly\overline{v}$.
\item
\label{c:nyc2010-10-31-aiii}
Suppose that $A_1$ is uniformly monotone at $\overline{x}$.
Then $x_n\to\overline{x}$.
\item
\label{c:nyc2010-10-31-aiv}
Suppose that $A_2^{-1}$ is uniformly monotone at $\overline{v}$.
Then $v_n\to\overline{v}$.
\end{enumerate}
\end{corollary}
\begin{proof}
Apply Theorem~\ref{t:nyc2010-10-31} with $\GG=\HH$, $L=\Id$, 
$A=A_1$, $B=A_2$, $r=0$, and $z=0$.
\end{proof}

\begin{remark}
\label{r:nyc2010-10-31-1}
The most popular algorithm to find a zero of the sum of two 
maximally monotone operators is the Douglas-Rachford algorithm
\cite{Joca09,Ecks92,Lion79,Svai10} (see \eqref{e:7ans}). 
Corollary~\ref{c:nyc2010-10-31-1} provides an alternative scheme
which is also based on evaluations of the resolvents of the 
two operators.
\end{remark}

\begin{corollary}
\label{c:nyc2010-10-31-b}
In Problem~\ref{prob:1}, suppose that 
$L\neq 0$ and that $\zer(L^*BL)\neq\emp$.
Let $(a_{1,n})_{n\in\NN}$ and
$(c_{1,n})_{n\in\NN}$ be absolutely summable sequences in $\HH$,
and let $(a_{2,n})_{n\in\NN}$, $(b_{n})_{n\in\NN}$, and
$(c_{2,n})_{n\in\NN}$ be absolutely summable sequences in $\GG$.
Let $x_0\in\HH$, let $v_0\in\GG$, let 
$\varepsilon\in\left]0,1/(\|L\|+1)\right[$, 
let $(\gamma_n)_{n\in\NN}$ be a sequence in 
$[\varepsilon,(1-\varepsilon)/\|L\|\,]$, and set
\begin{equation}
\label{e:nypl2010-10-27bx}
(\forall n\in\NN)\quad 
\begin{array}{l}
\left\lfloor
\begin{array}{l}
s_{n}=\gamma_n(L^*v_n+a_{1,n})\\
y_{n}=v_n+\gamma_n(Lx_n+a_{2,n})\\
p_{n}=J_{\gamma_n B^{-1}}y_{n}+b_{n}\\
x_{n+1}=x_n-\gamma_n(L^*p_n+c_{1,n})\\
v_{n+1}=p_{n}-\gamma_n(Ls_{n}+c_{2,n}).
\end{array}
\right.\\
\end{array}
\end{equation}
Then the following hold for some $\overline{x}\in\zer(L^*BL)$ and 
some $\overline{v}\in(\ran L)^\bot\cap B(L\overline{x})$.
\begin{enumerate}
\item
\label{c:nyc2010-10-31-bi}
$x_n\weakly\overline{x}$ and $v_n\weakly\overline{v}$.
\item
\label{c:nyc2010-10-31-biv}
Suppose that $B^{-1}$ is uniformly monotone at $\overline{v}$.
Then $v_n\to\overline{v}$.
\end{enumerate}
\end{corollary}
\begin{proof}
Apply Theorem~\ref{t:nyc2010-10-31} with $A=0$, $r=0$, and $z=0$.
\end{proof}

\begin{remark}
\label{r:nyc2010-10-31-b}
In connection with Corollary~\ref{c:nyc2010-10-31-b}, a weakly 
convergent splitting method was proposed in \cite{Penn02} 
for finding a zero of $L^*BL$. This method requires the additional 
assumption that $\ran L$ be closed. In addition,
unlike the algorithm described in \eqref{e:nypl2010-10-27bx},
it requires the exact implementation of the generalized inverse 
of $L$ at each iteration, which is challenging task. 
\end{remark}

Next, we extend \eqref{e:primal} to the problem of solving an 
inclusion involving the sum of $m$ composite monotone operators.
We obtain an algorithm in which the operators 
$(B_i)_{1\leq i\leq m}$ can be activated in parallel, and
independently from the transformations $(L_i)_{1\leq i\leq m}$.

\begin{theorem}
\label{t:2010-11-11}
Let $z\in\HH$ and let $(\omega_i)_{1\leq i\leq m}$ be reals in 
$\left]0,1\right]$ such that $\sum_{i=1}^m\omega_i=1$. 
For every $i\in\{1,\ldots,m\}$, let 
$(\GG_i,\|\cdot\|_{\GG_i})$ be a real Hilbert space, 
let $r_i\in\GG_i$, let $B_i\colon\GG_i\to 2^{\GG_i}$ 
be maximally monotone, and suppose that $0\neq L_i\in\BL(\HH,\GG_i)$.
Moreover, assume that 
\begin{equation}
\label{e:2010-11-03}
z\in\ran\sum_{i=1}^m\omega_iL_i^*\circ B_i\circ (L_i\cdot-r_i).
\end{equation}
Consider the problem
\begin{equation}
\label{e:2010-11-11x}
\text{find}\quad x\in\HH\;\;\text{such that}\quad
z\in\sum_{i=1}^m\omega_iL_i^*B_i(L_ix-r_i),
\end{equation}
and the problem 
\begin{equation}
\label{e:eqdualprod}
\!\!\text{find}\!\!\quad \!\!v_1\in\GG_1,\ldots,v_m\in\GG_m
\quad\!\!\!
\text{such that}\!\!\quad\!\!\sum_{i=1}^m\omega_iL_i^*v_i=z
\!\!\quad\!\!\text{and}\quad
\!\!\!(\exi x\in\HH)\quad\!\!\!\!\!
\begin{cases}
v_1\!\in B_1(L_1x-r_1)\\
\hspace{0.5cm}\vdots\\
v_m\!\in B_m(L_mx-r_m).
\end{cases}
 \end{equation}
Now, for every $i\in\{1,\ldots,m\}$, let $(a_{1,i,n})_{n\in\NN}$ and
$(c_{1,i,n})_{n\in\NN}$ be absolutely summable sequences in $\HH$, 
let $(a_{2,i,n})_{n\in\NN}$, $(b_{i,n})_{n\in\NN}$, and 
$(c_{2,i,n})_{n\in\NN}$ be absolutely summable sequences in
$\GG_i$, let $x_{i,0}\in\HH$, and let $v_{i,0}\in\GG_i$.
Furthermore, set $\beta=\max_{1\leq i\leq m}\|L_i\|$, let 
$\varepsilon\in\left]0,1/(\beta+1)\right[$, 
let $(\gamma_n)_{n\in\NN}$ be a sequence in 
$[\varepsilon,(1-\varepsilon)/\beta]$, and set
\begin{equation}
\label{e:2010-11-11z}
(\forall n\in\NN)\quad 
\begin{array}{l}
\left\lfloor
\begin{array}{l}
x_{n}=\sum_{i=1}^m\omega_ix_{i,n}\\[2mm]
\operatorname{For}\;i=1,\ldots,m\\
\left\lfloor
\begin{array}{l}
y_{1,i,n}=x_{i,n}-\gamma_n(L_i^*v_{i,n}+a_{1,i,n})\\
y_{2,i,n}=v_{i,n}+\gamma_n(L_ix_{i,n}+a_{2,i,n})\\
\end{array}
\right.\\[5mm]
p_{1,n}=\sum_{i=1}^m\omega_iy_{1,i,n}+\gamma_nz\\[3mm]
\operatorname{For}\;i=1,\ldots,m\\
\left\lfloor
\begin{array}{l}
p_{2,i,n}=J_{\gamma_nB_i^{-1}}(y_{2,i,n}-\gamma_nr_i)+b_{i,n}\\
q_{1,i,n}=p_{1,n}-\gamma_n(L_i^*p_{2,i,n}+c_{1,i,n})\\
q_{2,i,n}=p_{2,i,n}+\gamma_n(L_ip_{1,n}+c_{2,i,n})\\
x_{i,n+1}=x_{i,n}-y_{1,i,n}+q_{1,i,n}\\
v_{i,n+1}=v_{i,n}-y_{2,i,n}+q_{2,i,n}.
\end{array}
\right.\\[1mm]
\end{array}
\right.\\
\end{array}
\end{equation}
Then the following hold for some solution $\overline{x}$ to 
\eqref{e:2010-11-11x} and some solution 
$(\overline{v}_i)_{1\leq i\leq m}$ to \eqref{e:eqdualprod} such 
that, for every $i\in\{1,\ldots,m\}$, 
$\overline{v}_i\in B_i(L_i\overline{x}-r_i)$.
\begin{enumerate}
\item
\label{t:2010-11-11i} 
$x_n\weakly\overline{x}$ and,
for every $i\in\{1,\ldots,m\}$, $v_{i,n}\weakly\overline{v}_i$.
\item
\label{t:2010-11-11iii} 
Suppose that, for every $i\in\{1,\ldots,m\}$,
$B_i^{-1}$ is strongly monotone at $\overline{v}_i$. Then, for every 
$i\in\{1,\ldots,m\}$, $v_{i,n}\to\overline{v}_i$.
\end{enumerate}
\end{theorem}
\begin{proof}
Let $\HHH$ be the real Hilbert space obtained by endowing 
the Cartesian product $\HH^m$ with the scalar product 
$\scal{\cdot}{\cdot}_{\HHH}\colon(\boldsymbol{x},\boldsymbol{y})
\mapsto\sum_{i=1}^m\omega_i\scal{x_i}{y_i}$, where 
$\boldsymbol{x}=(x_i)_{1\leq i\leq m}$ and 
$\boldsymbol{y}=(y_i)_{1\leq i\leq m}$ denote generic 
elements in $\HHH$. The associated norm is 
$\|\cdot\|_{\HHH}\colon\boldsymbol{x}\mapsto
\sqrt{\sum_{i=1}^m\omega_i\|x_i\|^2}$.
Likewise, let $\GGG$ denote the real Hilbert space obtained by 
endowing $\GG_1\times\cdots\times\GG_m$ with the scalar product 
and the associated norm respectively defined by
\begin{equation}
\label{e:palawan-mai2008-}
\scal{\cdot}{\cdot}_{\GGG}
\colon(\boldsymbol{y},\boldsymbol{z})\mapsto
\sum_{i=1}^m\omega_i\scal{y_i}{z_i}_{\GG_i}
\quad\text{and}\quad\|\cdot\|_{\GGG}\colon
\boldsymbol{y}\mapsto\sqrt{\sum_{i=1}^m\omega_i\|y_i\|_{\GG_i}^2}.
\end{equation}
Define 
\begin{equation}
\label{e:D}
\boldsymbol{V}=\menge{(x,\dots,x)\in\HHH}{x\in\HH}
\quad\text{and}\quad
{\boldsymbol j}\colon\HH\to{\boldsymbol V}
\colon x\mapsto(x,\ldots,x).
\end{equation}
In view of \eqref{e:normalcone}, the normal cone operator of
$\boldsymbol{V}$ is
\begin{equation}
\label{e:dorthog}
N_{\boldsymbol{V}}\colon\HHH\to 2^{\HHH}\colon\boldsymbol{x}\mapsto
\begin{cases}
\boldsymbol{V}^{\bot}=\menge{\boldsymbol{u}\in\HHH}
{\sum_{i=1}^m\omega_iu_i=0},
&\text{if}\;\;\boldsymbol{x}\in\boldsymbol{V};\\
\emp,&\text{otherwise.}
\end{cases}
\end{equation} 
Now set
\begin{equation}
\label{e:ABLrz}
{\boldsymbol A}=N_{\boldsymbol V},\;
{\boldsymbol B}\colon\GGG\to 2^{\GGG}\colon\boldsymbol{y}\mapsto 
\overset{m}{\underset{i=1}{\cart}}B_iy_i,\;
{\boldsymbol L}\colon\HHH\to\GGG\colon\boldsymbol{x}\mapsto
(L_ix_i)_{1\leq i\leq m},
\;\text{and}\;
\boldsymbol{r}=(r_i)_{1\leq i\leq m}.
\end{equation}
It is easily checked that $\boldsymbol{A}$ and $\boldsymbol{B}$
are maximally monotone with resolvents
\begin{equation}
\label{e:JC3}
(\forall\gamma\in\RPP)\quad 
J_{\gamma\boldsymbol A}\colon{\boldsymbol x}
\mapsto P_{\boldsymbol V}{\boldsymbol x}=
{\boldsymbol j}\bigg(\sum_{i=1}^m\omega_ix_i\bigg)
\quad\text{and}\quad
J_{\gamma\boldsymbol{B}^{-1}}\colon{\boldsymbol y}\mapsto
\big(J_{\gamma B_i^{-1}}y_i\big)_{1\leq i\leq m}.
\end{equation}
Moreover, $\boldsymbol{L}\in\BL(\HHH,\GGG)$ and
\begin{equation}
\label{e:L*}
\boldsymbol{L}^*\colon\GGG\to\HHH\colon\boldsymbol{v}\mapsto
\big(L_i^*v_i\big)_{1\leq i\leq m}.
\end{equation}
Now, set 
\begin{equation}
\label{e:2010-11-11k}
\begin{cases}
\boldsymbol{\mathcal P}=\menge{\boldsymbol{x}\in\HHH}
{\boldsymbol{j}(z)\in
\boldsymbol{A}\boldsymbol{x}+\boldsymbol{L}^*
\boldsymbol{B}(\boldsymbol{L}\boldsymbol{x}-\boldsymbol{r})}\\
\boldsymbol{\mathcal D}=\menge{\boldsymbol{v}\in\GGG}
{-\boldsymbol{r}\in-\boldsymbol{L}\boldsymbol{A}^{-1}
(\boldsymbol{j}(z)-\boldsymbol{L}^*\boldsymbol{v})
+\boldsymbol{B}^{-1}\boldsymbol{v}}.
\end{cases}
\end{equation}
Then, for every $x\in\HH$,
\begin{eqnarray}
\label{e:valparaiso2010-11-23a}
x\;\text{solves}\;\eqref{e:2010-11-11x}
&\Leftrightarrow&z\in\sum_{i=1}^m\omega_iL_i^*\big(B_i(L_ix-r_i)\big)
\nonumber\\
&\Leftrightarrow&\bigg(\exi(v_i)_{1\leq i\leq m}\in
\overset{m}{\underset{i=1}{\cart}}B_i(L_ix-r_i)\bigg)
\quad z=\sum_{i=1}^m\omega_iL_i^*v_i\nonumber\\
&\Leftrightarrow&\bigg(\exi(v_i)_{1\leq i\leq m}\in
\overset{m}{\underset{i=1}{\cart}}B_i(L_ix-r_i)\bigg)
\quad \sum_{i=1}^m\omega_i(z-L_i^*v_i)=0\nonumber\\
&\Leftrightarrow&\big(\exi\boldsymbol{v}\in
\boldsymbol{B}(\boldsymbol{L}\boldsymbol{j}(x)-\boldsymbol{r})\big)
\quad\boldsymbol{j}(z)-\boldsymbol{L}^*\boldsymbol{v}\in 
\boldsymbol{V}^\bot=\boldsymbol{A}\boldsymbol{j}(x)\nonumber\\
&\Leftrightarrow&\boldsymbol{j}(z)\in
\boldsymbol{A}\boldsymbol{j}(x)+\boldsymbol{L}^*
\boldsymbol{B}(\boldsymbol{L}\boldsymbol{j}(x)-\boldsymbol{r})
\nonumber\\
&\Leftrightarrow&\boldsymbol{j}(x)\in\boldsymbol{\mathcal P}
\subset\boldsymbol{V}.
\end{eqnarray}
Moreover, for every $\boldsymbol{v}\in\GGG$, 
\begin{eqnarray}
\label{e:valparaiso2010-11-23b}
\boldsymbol{v}\;\text{solves}\;\eqref{e:eqdualprod}
&\Leftrightarrow&
\sum_{i=1}^m\omega_i(z-L_i^*v_i)=0\quad\text{and}\quad
(\exi x\in\HH)\quad (v_i)_{1\leq i\leq m}\in
\overset{m}{\underset{i=1}{\cart}}B_i(L_ix-r_i)
\nonumber\\
&\Leftrightarrow&
(\exi x\in\HH)\quad\boldsymbol{j}(z)-\boldsymbol{L}^*\boldsymbol{v}
\in\boldsymbol{V}^\bot=\boldsymbol{A}\boldsymbol{j}(x)
\quad\text{and}\quad\boldsymbol{v}\in
\boldsymbol{B}(\boldsymbol{L}\boldsymbol{j}(x)-\boldsymbol{r})
\nonumber\\
&\Leftrightarrow&
(\exi x\in\HH)\quad\boldsymbol{j}(x)\in\boldsymbol{A}^{-1}\big
(\boldsymbol{j}(z)-\boldsymbol{L}^*\boldsymbol{v}\big)\quad
\text{and}\quad\boldsymbol{L}\boldsymbol{j}(x)-\boldsymbol{r}\in
\boldsymbol{B}^{-1}\boldsymbol{v}
\nonumber\\
&\Leftrightarrow&
(\exi \boldsymbol{x}\in\boldsymbol{V}=\dom\boldsymbol{A})\quad
\boldsymbol{x}\in\boldsymbol{A}^{-1}\big
(\boldsymbol{j}(z)-\boldsymbol{L}^*\boldsymbol{v}\big)\quad
\text{and}\quad\boldsymbol{L}\boldsymbol{x}-\boldsymbol{r}\in
\boldsymbol{B}^{-1}\boldsymbol{v}
\nonumber\\
&\Leftrightarrow&
-\boldsymbol{r}\in-\boldsymbol{L}
\boldsymbol{A}^{-1}\big(\boldsymbol{j}(z)-\boldsymbol{L}^*
\boldsymbol{v}\big)+\boldsymbol{B}^{-1}\boldsymbol{v}
\nonumber\\
&\Leftrightarrow&
\boldsymbol{v}\in\boldsymbol{\mathcal D}.
\end{eqnarray}
Altogether, solving the inclusion \eqref{e:2010-11-11x} in $\HH$ 
is equivalent to solving the inclusion $\boldsymbol{j}(z)\in
\boldsymbol{A}\boldsymbol{x}+\boldsymbol{L}^*\boldsymbol{B}
(\boldsymbol{L}\boldsymbol{x}-\boldsymbol{r})$ in $\HHH$
and solving \eqref{e:eqdualprod}
in $\GGG$ is equivalent to solving $-\boldsymbol{r}\in
\boldsymbol{B}^{-1}\boldsymbol{v}-\boldsymbol{L}
\boldsymbol{A}^{-1}\big(\boldsymbol{j}(z)-\boldsymbol{L}^*
\boldsymbol{v}\big)$ in $\GGG$. 
Next, let us show that the algorithm
described in \eqref{e:2010-11-11z} is a particular case of the
algorithm described in \eqref{e:nypl2010-10-27b} in 
Theorem~\ref{t:nyc2010-10-31}. To this end define,
for every $n\in\NN$,
$\boldsymbol{x}_n=(x_{i,n})_{1\leq i\leq m}$,
$\boldsymbol{v}_n=(v_{i,n})_{1\leq i\leq m}$,
$\boldsymbol{y}_{1,n}=(y_{1,i,n})_{1\leq i\leq m}$,
$\boldsymbol{y}_{2,n}=(y_{2,i,n})_{1\leq i\leq m}$,
$\boldsymbol{p}_{1,n}=\boldsymbol{j}(p_{1,n})$,
$\boldsymbol{p}_{2,n}=(p_{2,i,n})_{1\leq i\leq m}$,
$\boldsymbol{q}_{1,n}=(q_{1,i,n})_{1\leq i\leq m}$,
$\boldsymbol{q}_{2,n}=(q_{2,i,n})_{1\leq i\leq m}$,
$\boldsymbol{a}_{1,n}=(a_{1,i,n})_{1\leq i\leq m}$,
$\boldsymbol{a}_{2,n}=(a_{2,i,n})_{1\leq i\leq m}$,
$\boldsymbol{b}_{2,n}=(b_{i,n})_{1\leq i\leq m}$,
$\boldsymbol{c}_{1,n}=(c_{1,i,n})_{1\leq i\leq m}$, and
$\boldsymbol{c}_{2,n}=(c_{2,i,n})_{1\leq i\leq m}$.
Then we deduce from \eqref{e:ABLrz}, \eqref{e:JC3}, 
and \eqref{e:L*} that, in terms of these new variables, 
\eqref{e:2010-11-11z} can be rewritten as
\begin{equation}
\label{e:nypl2010-10-27bp}
(\forall n\in\NN)\quad 
\begin{array}{l}
\left\lfloor
\begin{array}{l}
\boldsymbol{y}_{1,n}=\boldsymbol{x}_n-\gamma_n(\boldsymbol{L}^*
\boldsymbol{v}_n+\boldsymbol{a}_{1,n})\\
\boldsymbol{y}_{2,n}=\boldsymbol{v}_n+\gamma_n(\boldsymbol{L}
\boldsymbol{x}_n+\boldsymbol{a}_{2,n})\\
\boldsymbol{p}_{1,n}=J_{\gamma_n \boldsymbol{A}}
(\boldsymbol{y}_{1,n}+\gamma_nz)\\
\boldsymbol{p}_{2,n}=J_{\gamma_n \boldsymbol{B}^{-1}}
(\boldsymbol{y}_{2,n}-\gamma_nr)+\boldsymbol{b}_{2,n}\\
\boldsymbol{q}_{1,n}=\boldsymbol{p}_{1,n}-\gamma_n
(\boldsymbol{L}^*\boldsymbol{p}_{2,n}+\boldsymbol{c}_{1,n})\\
\boldsymbol{q}_{2,n}=\boldsymbol{p}_{2,n}+\gamma_n
(\boldsymbol{L}\boldsymbol{p}_{1,n}+\boldsymbol{c}_{2,n})\\
\boldsymbol{x}_{n+1}=\boldsymbol{x}_n-\boldsymbol{y}_{1,n}+
\boldsymbol{q}_{1,n}\\
\boldsymbol{v}_{n+1}=\boldsymbol{v}_n-\boldsymbol{y}_{2,n}+
\boldsymbol{q}_{2,n}.
\end{array}
\right.\\
\end{array}
\end{equation}
Moreover,
$\|\boldsymbol{L}\|\leq\max_{1\leq i\leq m}\|L_i\|=\beta$, and 
our assumptions imply that the sequences 
$(\boldsymbol{a}_{1,n})_{n\in\NN}$, 
$(\boldsymbol{c}_{1,n})_{n\in\NN}$,
$(\boldsymbol{a}_{2,n})_{n\in\NN}$, 
$(\boldsymbol{b}_{2,n})_{n\in\NN}$,
and $(\boldsymbol{c}_{2,n})_{n\in\NN}$ are absolutely summable.
Furthermore, \eqref{e:2010-11-03} and
\eqref{e:valparaiso2010-11-23a} assert that 
$\boldsymbol{j}(z)\in\ran(\boldsymbol{A}+\boldsymbol{L}^*\circ
\boldsymbol{B}\circ(\boldsymbol{L}\cdot-\boldsymbol{r}))$.

\ref{t:2010-11-11i}: 
It follows from 
Theorem~\ref{t:nyc2010-10-31}\ref{t:nyc2010-10-31i} that there 
exists $\overline{\boldsymbol{x}}\in\boldsymbol{\mathcal P}$ 
and $(\overline{v}_i)_{1\leq i\leq m}=
\overline{\boldsymbol{v}}\in\boldsymbol{\mathcal D}$ 
such that
$\boldsymbol{j}(z)-\boldsymbol{L}^*\overline{\boldsymbol{v}}\in
\boldsymbol{A}\overline{\boldsymbol{x}}$,
$\overline{\boldsymbol{v}}\in\boldsymbol{B}(\boldsymbol{L}
\overline{\boldsymbol{x}}-\boldsymbol{r})$, 
$\boldsymbol{x}_{n}\weakly\overline{\boldsymbol{x}}$, and
$\boldsymbol{v}_{n}\weakly\overline{\boldsymbol{v}}$.
Hence, 
$\boldsymbol{j}(x_n)=P_{\boldsymbol{V}}\boldsymbol{x}_{n}\weakly
P_{\boldsymbol{V}}\overline{\boldsymbol{x}}=
\overline{\boldsymbol{x}}$. Since 
\eqref{e:valparaiso2010-11-23a} asserts that there 
exists a solution $\overline{x}$ to \eqref{e:2010-11-11x} such that 
$\overline{\boldsymbol{x}}=\boldsymbol{j}(\overline{x})$, 
we obtain that 
$x_{n}=\boldsymbol{j}^{-1}(P_{\boldsymbol{V}}\boldsymbol{x}_{n})
\weakly\boldsymbol{j}^{-1}(\overline{\boldsymbol{x}})=\overline{x}$.
Altogether, by \eqref{e:valparaiso2010-11-23b}, 
for every $i\in\{1,\ldots,m\}$, $v_{i,n}\weakly\overline{v}_i$,
where $(\overline{v}_i)_{1\leq i\leq m}$ solves 
\eqref{e:eqdualprod}. 

\ref{t:2010-11-11iii}: 
Let $(\boldsymbol{w}_1,\boldsymbol{y}_1)$ and 
$(\boldsymbol{w}_2,\boldsymbol{y}_2)$ in $\gr\boldsymbol{B}^{-1}$.
We derive from  \eqref{e:ABLrz} that
$(\forall i\in\{1,\ldots,m\})$ $y_{1,i}\in B_i^{-1}w_{1,i}$
and $y_{2,i}\in B_i^{-1}w_{2,i}$.
Hence, since the operators $(B_i^{-1})_{1\leq i\leq m}$ are 
strongly monotone, there exist constants $(\rho_i)_{1\leq i\leq m}$
in $\RPP$ such that 
$\scal{\boldsymbol{y}_1-\boldsymbol{y}_2}{\boldsymbol{w}_1-
\boldsymbol{w}_2}_{\GGG}=\sum_{i=1}^m\omega_i
\scal{y_{1,i}-y_{2,i}}{w_{1,i}-w_{2,i}}_{\GG_i}
\geq\sum_{i=1}^m\omega_i\rho_i\|w_{1,i}-w_{2,i}\|_{\GG_i}^2
\geq\rho\|\boldsymbol{w}_1-\boldsymbol{w}_2\|_{\GGG}^2$,
where $\rho=\min_{1\leq i\leq m}\rho_i\in\RPP$. Therefore, 
$\boldsymbol{B}^{-1}$ is strongly monotone and hence 
uniformly monotone. Thus, the result follows from 
Theorem~\ref{t:nyc2010-10-31}\ref{t:nyc2010-10-31iv}.
\end{proof}

\section{Variational problems}
\label{sec:4}

We apply the results of the previous sections to
minimization problems. Let us first recall some standard notation 
and results \cite{Livre1,Zali02}. 
We denote by $\Gamma_0(\HH)$ the class 
of lower semicontinuous convex functions $f\colon\HH\to\RX$ such that
$\dom f=\menge{x\in\HH}{f(x)<\pinf}\neq\emp$. Now let
$f\in\Gamma_0(\HH)$. The conjugate of $f$ is the function 
$f^*\in\Gamma_0(\HH)$ defined by 
$f^*\colon u\mapsto\sup_{x\in\HH}(\scal{x}{u}-f(x))$. 
Moreover, for every $x\in\HH$, $f+\|x-\cdot\|^2/2$ possesses a 
unique minimizer, which is denoted by $\prox_fx$.
Alternatively, 
\begin{equation}
\label{e:prox2}
\prox_f=(\Id+\partial f)^{-1}=J_{\partial f},
\end{equation}
where $\partial f\colon\HH\to 2^{\HH}\colon x\mapsto
\menge{u\in\HH}{(\forall y\in\HH)\;\:\scal{y-x}{u}+f(x)\leq f(y)}$ 
is the subdifferential of $f$, which is a maximally monotone 
operator. Finally, let $C$ be a convex subset of $\HH$. 
The indicator function of $C$ is denoted by $\iota_C$, 
its support function by $\sigma_C$, and its strong relative interior 
(the set of points in $x\in C$ such that the cone 
generated by $-x+C$ is a closed vector subspace of $\HH$)
by $\sri C$.
The following facts will also be required.

\begin{proposition}
\label{p:cq}
Let $f\in\Gamma_0(\HH)$, let $g\in\Gamma_0(\GG)$,
let $L\in\BL(\HH,\GG)$, let $z\in\HH$, and let $r\in\GG$. 
Then the following hold. 
\begin{enumerate}
\item
\label{p:cqi}
$\zer(-z+\partial f+L^*\circ(\partial g)\circ(L\cdot -r))
\subset\Argmin(f-\scal{\cdot}{z}+g\circ(L\cdot-r))$.
\item
\label{p:cqii}
$\zer(r-(L\circ(\partial f^*)\circ(z-L^*\cdot))
+\partial g^*)\subset\Argmin(f^*(z-L^*\cdot)+g^*+\scal{r}{\cdot})$.
\item
\label{p:cqiii}
Suppose that one of the following is satisfied.
\begin{enumerate}
\item
\label{p:cqiiia}
$\Argmin(f+g\circ(L\cdot-r)-\scal{\cdot}{z})\neq\emp$ and 
$r\in\sri(L(\dom f)-\dom g)$.
\item
\label{p:cqiiib}
$\Argmin(f+g\circ(L\cdot-r)-\scal{\cdot}{z})
\subset\Argmin (f-\scal{\cdot}{z})
\cap\Argmin g\circ(L\cdot-r)\neq\emp$
and $r\in \sri(\ran L-\dom g)$.
\item
\label{p:cqiiic}
$f=\iota_C$ and $g=\iota_D$, $z=0$, where $C$ and $D$ are closed 
convex subset of $\HH$ and $\GG$, respectively, such that 
$C\cap L^{-1}(r+D)\neq\emp$ and $r\in\sri(\ran L-D)$.
\end{enumerate}
Then $z\in\ran(\partial f+L^*\circ(\partial g)\circ (L\cdot-r))$.
\end{enumerate}
\end{proposition}
\begin{proof}
\ref{p:cqi}\&\ref{p:cqii}: By
\cite[Proposition~16.5(ii) and Theorem~16.2]{Livre1},
$\zer(-z+\partial f+L^*\circ(\partial g)\circ(L\cdot -r))
\subset\zer(\partial(f-\scal{\cdot}{z}+g\circ(L\cdot-r)))
=\Argmin(f-\scal{\cdot}{z}+g\circ(L\cdot-r))$. 
We obtain \ref{p:cqii} similarly.

\ref{p:cqiiia}:
By \cite[Theorem~16.2 and Theorem~16.37(i)]{Livre1}, we have 
\begin{alignat}{2}
\emp\neq \Argmin(f + g\circ (L\cdot -r) - \scal{\cdot}{z})
&=\zer\partial(f + g\circ (L\cdot -r) - \scal{\cdot}{z})\nonumber\\
&=\zer(-z+\partial f +  L^*\circ(\partial g)\circ(L\cdot-r)).
\end{alignat}
 
\ref{p:cqiiib}:
Since $r\in\sri(\ran L-\dom g)$, using \ref{p:cqi} and standard 
convex analysis, we obtain
\begin{align}
\Argmin (f-\scal{\cdot}{z})\cap\Argmin(g\circ(L\cdot-r))
&=\zer(-z+\partial f)\cap\zer\partial(g\circ(L\cdot-r))\nonumber\\
&=\zer(-z+\partial f)\cap\zer(L^*\circ(\partial g)\circ(L\cdot-r))
\nonumber\\
&\subset\zer(-z+\partial f +  L^*\circ(\partial g)\circ(L\cdot-r))
\nonumber\\
&\subset\Argmin(f+g\circ(L\cdot-r)-\scal{\cdot}{z}).
\end{align}
Therefore, the hypotheses yield
$\zer(-z+\partial f+L^*\circ(\partial g)\circ(L\cdot-r))=
\Argmin(f-\scal{\cdot}{z})\cap\Argmin(g\circ(L\cdot-r))\neq\emp$.

\ref{p:cqiiic}: 
Since $\dom(\iota_{C}+\iota_{D}(L\cdot-r))=C\cap L^{-1}(r+ D)$,
\begin{align}
\Argmin(\iota_C+\iota_D\circ(L\cdot-r))
&=\Argmin\iota_{C \cap L^{-1}(r+ D)}\nonumber\\
&=C\cap L^{-1}(r+ D)\nonumber\\
&=\Argmin\iota_C \cap \Argmin(\iota_D\circ(L\cdot-r))\neq\emp.
\end{align}
In view of \ref{p:cqii} applied to $f=\iota_C$, $g=\iota_D$, and 
$z=0$, the proof is complete.
\end{proof}

Our first result is a new splitting method for the 
Fenchel-Rockafellar duality framework 
\eqref{e:2010-11-18f}--\eqref{e:2010-11-18g}.

\begin{proposition}
\label{p:2010-11-12}
Let $f\in\Gamma_0(\HH)$, let $g\in\Gamma_0(\GG)$,
let $L\in\BL(\HH,\GG)$, let $z\in\HH$, 
and let $r\in\GG$. Suppose that 
$L\neq 0$ and that
\begin{equation}
\label{e:2010-11-13e}
z\in\ran\big(\partial f+L^*\circ(\partial g) \circ (L\cdot-r)\big).
\end{equation}
Consider the primal problem
\begin{equation}
\label{e:2010-11-13f}
\minimize{x\in\HH}{f(x)+g(Lx-r)-\scal{x}{z}},
\end{equation}
and the dual problem
\begin{equation}
\label{e:2010-11-13g}
\minimize{v\in\GG}{f^*(z-L^*v)+g^*(v)+\scal{v}{r}}.
\end{equation}
Let $(a_{1,n})_{n\in\NN}$, $(b_{1,n})_{n\in\NN}$, and
$(c_{1,n})_{n\in\NN}$ be absolutely summable sequences in $\HH$,
and let $(a_{2,n})_{n\in\NN}$, $(b_{2,n})_{n\in\NN}$, and
$(c_{2,n})_{n\in\NN}$ be absolutely summable sequences in $\GG$.
Furthermore, let $x_0\in\HH$, let $v_0\in\GG$, let 
$\varepsilon\in\left]0,1/(\|L\|+1)\right[$, 
let $(\gamma_n)_{n\in\NN}$ be a sequence in 
$[\varepsilon,(1-\varepsilon)/\|L\|\,]$, and set
\begin{equation}
\label{e:2010-11-13d}
(\forall n\in\NN)\quad 
\begin{array}{l}
\left\lfloor
\begin{array}{l}
y_{1,n}=x_n-\gamma_n(L^*v_n+a_{1,n})\\
y_{2,n}=v_n+\gamma_n(Lx_n+a_{2,n})\\
p_{1,n}=\prox_{\gamma_n f}(y_{1,n}+\gamma_nz)+b_{1,n}\\
p_{2,n}=\prox_{\gamma_n g^*}(y_{2,n}-\gamma_nr)+b_{2,n}\\
q_{1,n}=p_{1,n}-\gamma_n(L^*p_{2,n}+c_{1,n})\\
q_{2,n}=p_{2,n}+\gamma_n(Lp_{1,n}+c_{2,n})\\
x_{n+1}=x_n-y_{1,n}+q_{1,n}\\
v_{n+1}=v_n-y_{2,n}+q_{2,n}.
\end{array}
\right.\\
\end{array}
\end{equation}
Then the following hold for some solution $\overline{x}$ to 
\eqref{e:2010-11-13f} and some solution $\overline{v}$ to 
\eqref{e:2010-11-13g} such that 
$z-L^*\overline{v}\in\partial f(\overline{x})$ and 
$\overline{v}\in\partial g(L\overline{x}-r)$.
\begin{enumerate}
\item
\label{p:2010-11-12i-}
$x_n-p_{1,n}\to 0$ and $v_n-p_{2,n}\to 0$.
\item
\label{p:2010-11-12i}
$x_n\weakly\overline{x}$, $p_{1,n}\weakly\overline{x}$,
$v_n\weakly\overline{v}$, and $p_{2,n}\weakly\overline{v}$.
\item
\label{p:2010-11-12iii} 
Suppose that $f$ is uniformly convex at $\overline{x}$.
Then $x_n\to\overline{x}$ and $p_{1,n}\to\overline{x}$.
\item
\label{p:2010-11-12iv}
Suppose that $g^{*}$ is uniformly convex at $\overline{v}$.
Then $v_n\to\overline{v}$ and $p_{2,n}\to\overline{v}$.
\end{enumerate}
\end{proposition}
\begin{proof}
Suppose that $A=\partial f$ and $B=\partial g$ in
Problem~\ref{prob:1}. Then, since $A^{-1}=\partial f^*$
and $B^{-1}=\partial g^*$,
we derive from Proposition~\ref{p:cq}\ref{p:cqi}\&\ref{p:cqii} 
that the solutions to \eqref{e:primal} and \eqref{e:dual} are 
solutions to 
\eqref{e:2010-11-13f} and \eqref{e:2010-11-13g}, respectively. 
Moreover, \eqref{e:prox2} implies that \eqref{e:2010-11-13d} is 
a special case of \eqref{e:nypl2010-10-27b}. Finally, the uniform
convexity of a function $\varphi\in\Gamma_0(\HH)$ at a point of 
the domain of $\partial\varphi$ implies the uniform monotonicity 
of $\partial\varphi$ at that point \cite[Section~3.4]{Zali02}.
Altogether, the results follow from Theorem~\ref{t:nyc2010-10-31}. 
\end{proof}

\begin{remark}
\label{r:2010-11-12}
Here are some comments on Proposition~\ref{p:2010-11-12}.
\begin{enumerate}
\item
Sufficient conditions for \eqref{e:2010-11-13e} to hold are 
provided in Proposition~\ref{p:cq}.
\item
As in Remark~\ref{r:1}, if the proximity operator of $g$ is 
simpler to implement than that of $g^*$, $p_{2,n}$ in 
\eqref{e:2010-11-13d} can be computed via the identity
$\prox_{\gamma_n g^*}y=y-\gamma_n \prox_{\gamma_n^{-1}g}
(\gamma_n^{-1}y)$.
\item
In the special case when $\HH$ and $\GG$ are Euclidean spaces,
an alternative primal-dual algorithm is proposed in \cite{Chen94},
which also uses the proximity operators of $f$ and $g$, and the
operator $L$ in separate steps. This method is derived there in the
spirit of the proximal \cite{Roc76b} and alternating direction (see 
\cite{Fort83} and the references therein)
methods of multipliers.
\end{enumerate}
\end{remark}

We now turn our attention to problems involving
the sum of $m$ composite functions.

\begin{proposition}
\label{p:4}
Let $z\in\HH$ and let $(\omega_i)_{1\leq i\leq m}$ be reals in 
$\left]0,1\right]$ such that $\sum_{i=1}^m\omega_i=1$.
For every $i\in\{1,\ldots,m\}$, let 
$(\GG_i,\|\cdot\|_{\GG_i})$ be a real Hilbert space, 
let $r_i\in\GG_i$, let $g_i\in\Gamma_0(\GG_i)$, and suppose that
$0\neq L_i\in\BL(\HH,\GG_i)$. Moreover, assume that
\begin{equation}
\label{e:2010-11-03I}
z\in\ran\sum_{i=1}^m\omega_iL_i^*\circ(\partial g_i)
\circ(L_i\cdot-r_i).
\end{equation}
Consider the problem
\begin{equation}
\label{e:primalvar}
\minimize{x\in\HH}{\sum_{i=1}^m\omega_i\,g_i(L_ix-r_i)-
\scal{x}{z}}, 
\end{equation}
and the problem
\begin{equation}
\label{e:dualvar}
\minimize{\substack{v_1\in\GG_1,\ldots,v_m\in\GG_m\\
\sum_{i=1}^m\omega_iL_i^*v_i=z}}
{\sum_{i=1}^m\omega_i\,\big(g_i^*(v_i)+
\scal{v_i}{r_i}\big)}.
\end{equation}
For every $i\in\{1,\ldots,m\}$, let $(a_{1,i,n})_{n\in\NN}$ and 
$(c_{1,i,n})_{n\in\NN}$ be absolutely summable sequences in $\HH$, 
let $(a_{2,i,n})_{n\in\NN}$, $(b_{i,n})_{n\in\NN}$, and 
$(c_{2,i,n})_{n\in\NN}$ be absolutely summable sequences in
$\GG_i$, let $x_{i,0}\in\HH$, and let $v_{i,0}\in\GG_i$.
Furthermore, set $\beta=\max_{1\leq i\leq m}\|L_i\|$, let 
$\varepsilon\in\left]0,1/(\beta+1)\right[$, 
let $(\gamma_n)_{n\in\NN}$ be a sequence in 
$[\varepsilon,(1-\varepsilon)/\beta]$, and set
\begin{equation}
\label{e:2010-11-11var}
(\forall n\in\NN)\quad 
\begin{array}{l}
\left\lfloor
\begin{array}{l}
x_{n}=\sum_{i=1}^m\omega_ix_{i,n}\\
\operatorname{For}\;i=1,\ldots,m\\
\left\lfloor
\begin{array}{l}
y_{1,i,n}=x_{i,n}-\gamma_n(L_i^*v_{i,n}+a_{1,i,n})\\
y_{2,i,n}=v_{i,n}+\gamma_n(L_ix_{i,n}+a_{2,i,n})\\
\end{array}
\right.\\[3mm]
p_{1,n}=\sum_{i=1}^m\omega_iy_{1,i,n}+\gamma_nz\\[1mm]
\operatorname{For}\;i=1,\ldots,m\\
\left\lfloor
\begin{array}{l}
p_{2,i,n}=\prox_{\gamma_ng_i^*}(y_{2,i,n}-\gamma_nr_i)+b_{i,n}\\
q_{1,i,n}=p_{1,n}-\gamma_n(L_i^*p_{2,i,n}+c_{1,i,n})\\
q_{2,i,n}=p_{2,i,n}+\gamma_n(L_ip_{1,n}+c_{2,i,n})\\
x_{i,n+1}=x_{i,n}-y_{1,i,n}+q_{1,i,n}\\
v_{i,n+1}=v_{i,n}-y_{2,i,n}+q_{2,i,n}.
\end{array}
\right.\\[0mm]
\end{array}
\right.\\
\end{array}
\end{equation}
Then the following hold for some solution $\overline{x}$ to 
\eqref{e:primalvar} and some solution 
$(\overline{v}_i)_{1\leq i\leq m}$ to \eqref{e:dualvar} such that,
for every $i\in\{1,\ldots,m\}$, $\overline{v}_i\in\partial g_i
(L_i\overline{x}-r_i)$.
\begin{enumerate}
\item
\label{t:2010-11-11ivar} 
$x_n\weakly\overline{x}$ and,
for every $i\in\{1,\ldots,m\}$, $v_{i,n}\weakly\overline{v}_i$.
\item
\label{t:2010-11-11iiivar} 
Suppose that, for every $i\in\{1,\ldots,m\}$,
$g_i^*$ is strongly convex at $\overline{v}_i$.
Then, for every $i\in\{1,\ldots,m\}$, 
$v_{i,n}\to\overline{v}_i$.
\end{enumerate}
\end{proposition}
\begin{proof}
Define $\HHH$, $\GGG$, $\boldsymbol{L}$, $\boldsymbol{V}$, 
$\boldsymbol{r}$, and $\boldsymbol{j}\colon\HH\to\boldsymbol{V}$ 
as in the proof of Theorem~\ref{t:2010-11-11}.
Moreover, set $\boldsymbol{f}=\iota_{\boldsymbol{V}}$ and 
$\boldsymbol{g}\colon\GGG\to\RX\colon\boldsymbol{y}\mapsto
\sum_{i=1}^m\omega_i\,g_i(y_i)$.
Then, $\boldsymbol{f}\in\Gamma_0(\HHH)$, 
$\boldsymbol{g}\in\Gamma_0(\GGG)$, 
$\boldsymbol{f}^*=\iota_{\boldsymbol{V}^{\bot}}$, 
and $\boldsymbol{g}^*\colon\boldsymbol{v}\mapsto
\sum_{i=1}^m\omega_i\,g_i^*(v_i)$. Therefore, 
\eqref{e:2010-11-03I} is equivalent to
\begin{equation}
\boldsymbol{j}(z)\in\ran
\big(\partial\boldsymbol{f}+\boldsymbol{L}^*\circ
(\partial\boldsymbol{g})\circ(\boldsymbol{L}\cdot-
\boldsymbol{r})\big).
\end{equation}
Furthermore, \eqref{e:primalvar} and \eqref{e:dualvar} are 
equivalent to
\begin{equation}
\label{e:2010-11-13fvar}
\minimize{\boldsymbol{x}\in\HHH}
{\boldsymbol{f}(\boldsymbol{x})+\boldsymbol{g}
(\boldsymbol{L}\boldsymbol{x}-\boldsymbol{r})-
\scal{\boldsymbol{x}}{\boldsymbol{j}(z)}}_{\HHH}
\end{equation}
and
\begin{equation}
\label{e:2010-11-13gvar}
\minimize{\boldsymbol{v}\in\GGG}
{\boldsymbol{f}^*(\boldsymbol{j}(z)-
\boldsymbol{L}^*\boldsymbol{v})+\boldsymbol{g}^*(\boldsymbol{v})
+\scal{\boldsymbol{r}}{\boldsymbol{v}}}_{\GGG},
\end{equation}
respectively. On the other hand since, for every $\gamma\in\RPP$,
$\prox_{\gamma\boldsymbol{f}}\colon
\boldsymbol{x}\mapsto{\boldsymbol j}(\sum_{i=1}^m\omega_ix_i)$ and 
$\prox_{\gamma\boldsymbol{g}^*}=
(\prox_{\gamma g_i^*})_{1\leq i\leq m}$, 
\eqref{e:2010-11-11var} is a particular case of 
\eqref{e:2010-11-13d}. Finally, in 
\ref{t:2010-11-11iiivar}, $\boldsymbol{g}^*$ is strongly, hence 
uniformly, convex at $\overline{\boldsymbol{v}}$. Altogether, 
the results follow from Proposition~\ref{p:2010-11-12}.
\end{proof}

\begin{remark}
Suppose that \eqref{e:primalvar} has a solution and that
\begin{equation}
\label{e:condqualprod}
(r_1,\ldots,r_m)\in\sri\menge{(L_1x-y_1,\ldots,L_mx-y_m)}
{x\in\HH,\,y_1\in\dom g_1,\ldots,\,y_m\in\dom g_m}. 
\end{equation}
Then, with the notation of the proof of 
Proposition~\ref{p:4}, \eqref{e:condqualprod} is 
equivalent to
$\boldsymbol{r}\in
\sri(\boldsymbol{L}(\boldsymbol{V})-\dom\boldsymbol{g})
=\sri(\boldsymbol{L}(\dom\boldsymbol{f})-\dom\boldsymbol{g})$.
Thus, Proposition~\ref{p:cq}\ref{p:cqiiia} 
asserts that \eqref{e:2010-11-03I} holds. 
\end{remark}


\begin{thebibliography}{99}
\setlength{\itemsep}{1pt} 

\small

\bibitem{Aldu05}  
G. Alduncin,
Composition duality principles for mixed variational inequalities,
{\em Math. Comput. Modelling,}
vol. 41, pp. 639--654, 2005.

\bibitem{Sico10}
H. Attouch, L. M. Brice\~no-Arias, and P. L. Combettes,
A parallel splitting method for coupled monotone inclusions,
{\em SIAM J. Control Optim.},
vol. 48, pp. 3246--3270, 2010. 

\bibitem{Atto96}  
H. Attouch and M. Th\'era,
A general duality principle for the sum of two operators,
{\em J. Convex Anal.,}
vol. 3, pp. 1--24, 1996. 

\bibitem{Ther96}  
H. Attouch and M. Th\'era,
A duality proof of the Hille-Yosida theorem,
in: {\em Progress in Partial Differential Equations: the Metz
Surveys,} vol. 4, pp. 18--35, 1996.

\bibitem{Moor01} 
H. H. Bauschke and P. L. Combettes,
A weak-to-strong convergence principle for Fej\'er-monotone methods 
in Hilbert spaces,
{\em Math. Oper. Res.,} 
vol. 26, pp. 248--264, 2001.

\bibitem{Livre1} 
H. H. Bauschke and P. L. Combettes,
{\em Convex Analysis and Monotone Operator Theory in Hilbert Spaces.}
Springer-Verlag, New York, 2011.

\bibitem{Reic05}
H. H. Bauschke, P. L. Combettes, and S. Reich,
The asymptotic behavior of the composition of two resolvents,
{\em Nonlinear Anal.},
vol. 60, pp. 283--301, 2005.

\bibitem{Chen94}
G. Chen and M. Teboulle,
A proximal-based decomposition method for convex 
minimization problems,
{\em Math. Programming,}
vol. 64, pp. 81--101, 1994.

\bibitem{Else01} 
P. L. Combettes,
Quasi-Fej\'erian analysis of some optimization algorithms, in:
{\em Inherently Parallel Algorithms for Feasibility and 
Optimization},
D. Butnariu, Y. Censor, S. Reich (eds.).
Elsevier, Amsterdam, pp. 115--152, 2001. 

\bibitem{Opti04}
P. L. Combettes, 
Solving monotone inclusions via compositions of nonexpansive 
averaged operators,
{\em Optimization,}
vol. 53, pp. 475--504, 2004.

\bibitem{Joca09}
P. L. Combettes, 
Iterative construction of the resolvent of a sum of 
maximal monotone operators, 
{\em J. Convex Anal.,}
vol. 16, pp. 727--748, 2009.

\bibitem{Banf09}
P. L. Combettes and J.-C. Pesquet,
Proximal splitting methods in signal processing, in: 
{\em Fixed-Point Algorithms for Inverse Problems in Science 
and Engineering,}
H. H. Bauschke {\em et al.} (eds.).
Springer-Verlag, New York, 2011.

\bibitem{Smms05}
P. L. Combettes and V. R. Wajs,
Signal recovery by proximal forward-backward splitting,
{\em Multiscale Model. Simul.},
vol. 4, pp.~1168--1200, 2005. 

\bibitem{Ecks92} 
J. Eckstein and D. P. Bertsekas, 
On the Douglas-Rachford splitting method and the proximal 
point algorithm for maximal monotone operators,
{\em Math. Programming,}
vol. 55, pp. 293--318, 1992.

\bibitem{Ecks99} 
J. Eckstein and M. C. Ferris,
Smooth methods of multipliers for complementarity problems,
{\em Math. Programming,}
vol. 86, pp. 65--90, 1999.

\bibitem{Svai08} 
J. Eckstein and B. F. Svaiter, 
A family of projective splitting methods for the sum of two
maximal monotone operators,
{\em Math. Programming,}
vol. 111, pp. 173--199, 2008.

\bibitem{Ekel99} 
I. Ekeland and R. Temam,
{\em Analyse Convexe et Probl\`emes Variationnels,}
Dunod, Paris, 1974; 
{\em Convex Analysis and Variational Problems,}
SIAM, Philadelphia, PA, 1999.

\bibitem{Facc03} 
F. Facchinei and J.-S. Pang,
{\em Finite-Dimensional Variational Inequalities and 
Complementarity Problems.}
Springer-Verlag, New York, 2003.

\bibitem{Fuku96} 
M. Fukushima, 
The primal Douglas-Rachford splitting algorithm for a class of 
monotone mappings with applications to the traffic equilibrium 
problem,
{\em Math. Programming,}
vol. 72, pp. 1--15, 1996.

\bibitem{Fort83} 
M. Fortin and R. Glowinski (eds.), 
{\em Augmented Lagrangian Methods: Applications to the 
Numerical Solution of Boundary Value Problems.}
North-Holland, Amsterdam, 1983.

\bibitem{Gaba83} 
D. Gabay,
Applications of the method of multipliers to variational 
inequalities, in: M. Fortin and R. Glowinski (eds.), 
{\em Augmented Lagrangian Methods: Applications to the 
Numerical Solution of Boundary Value Problems,}
pp. 299--331. North-Holland, Amsterdam, 1983.

\bibitem{Glow89}
R. Glowinski and P. Le Tallec (eds.),
{\em Augmented Lagrangian and Operator-Splitting
 Methods in Nonlinear Mechanics.}
SIAM, Philadelphia, 1989.

\bibitem{Lion79} 
P.-L. Lions and B. Mercier, 
Splitting algorithms for the sum of two nonlinear operators,
{\em SIAM J. Numer. Anal.},
vol. 16,  pp. 964--979, 1979.

\bibitem{Mcli74} 
L. McLinden,
An extension of Fenchel's duality theorem to saddle functions
and dual minimax problems,
{\em Pacific J. Math.,}
vol. 50, pp. 135--158, 1974.

\bibitem{Merc79} 
B. Mercier, 
{\em Topics in Finite Element Solution of Elliptic Problems}
(Lectures on Mathematics, no. 63).
Tata Institute of Fundamental Research, Bombay, 1979.

\bibitem{Merc80} 
B. Mercier,  
{\em In\'equations Variationnelles de la M\'ecanique}
(Publications Math\'ematiques d'Orsay, no. 80.01).
Universit\'e de Paris-XI, Orsay, France, 1980. 

\bibitem{Mosc72} 
U. Mosco, 
Dual variational inequalities, 
{\em J. Math. Anal. Appl.,}
vol. 40, pp. 202--206, 1972.

\bibitem{Penn00} 
T. Pennanen, 
Dualization of generalized equations of maximal monotone type,
{\em SIAM J. Optim.,}
vol. 10, pp. 809--835, 2000.

\bibitem{Penn02} 
T. Pennanen,
A splitting method for composite mappings,
{\em Numer. Funct. Anal. Optim.},
vol. 23, pp. 875--890, 2002.

\bibitem{Robi99} 
S. M. Robinson, 
Composition duality and maximal monotonicity,
{\em Math. Programming,}
vol. 85, pp. 1--13, 1999.

\bibitem{Robi01} 
S. M. Robinson, 
Generalized duality in variational analysis, in:
N. Hadjisavvas and P. M. Pardalos (eds.),
{\em Advances in Convex Analysis and Global Optimization,}
pp. 205--219. Dordrecht, The Netherlands, Kluwer, 2001.

\bibitem{Rock67} 
R. T. Rockafellar,
Duality and stability in extremum problems involving convex 
functions,
{\em Pacific J. Math.,}
vol. 21, pp. 167--187, 1967. 

\bibitem{Roc76a} R. T. Rockafellar, 
Monotone operators and the proximal point algorithm,
{\em SIAM J. Control Optim.,}
vol. 14, pp. 877--898, 1976.

\bibitem{Roc76b}
R. T. Rockafellar,
Augmented Lagrangians and applications of the proximal point 
algorithm in convex programming,
{\it Math. Oper. Res.,}
vol. 1, 97--116, 1976.

\bibitem{Rock74} 
R. T. Rockafellar,
{\em Conjugate Duality and Optimization.}
SIAM, Philadelphia, PA, 1974.

\bibitem{Svai10} 
B. F. Svaiter, 
Weak convergence on Douglas-Rachford method,
{\em SIAM J. Control Optim.,}
to appear.

\bibitem{Tsen90} 
P. Tseng, Further applications of a splitting algorithm to
decomposition in variational inequalities and convex programming,
{\em Math. Programming,}
vol. 48, no. 2, pp. 249--263, 1990.

\bibitem{Tsen91} 
P. Tseng, Applications of a splitting algorithm to
decomposition in convex programming and variational inequalities,
{\em SIAM J. Control Optim.,}
vol. 29, pp. 119--138, 1991.

\bibitem{Tsen00} 
P. Tseng, 
A modified forward-backward splitting method for 
maximal monotone mappings, 
{\em SIAM J. Control Optim.}, 
vol. 38, pp. 431--446, 2000.

\bibitem{Zali02} 
C. Z\u{a}linescu,
{\em Convex Analysis in General Vector Spaces,}
World Scientific, River Edge, NJ, 2002.

\bibitem{ZeidXX} 
E. Zeidler,
{\em Nonlinear Functional Analysis and Its Applications,} 
vols. I--V.
Springer-Verlag, New York, 1985--1993.
\end{thebibliography}
\end{document}